\documentclass[reqno,a4paper,11pt]{amsart}

\numberwithin{equation}{section}
\usepackage[utf8]{inputenc}
\usepackage{amsmath, amsthm, amssymb, amscd, accents, bm, amsfonts}
\usepackage{mathtools}
\usepackage{url}
\usepackage{mathrsfs,dsfont}
\usepackage{mathtools}
\usepackage{mdwlist}
\usepackage{paralist}
\usepackage{datetime}
\usepackage{hyperref}
\usepackage{colonequals}
\usepackage{soul}
\usepackage[dvipsnames]{xcolor}
\usepackage{enumitem}
\setlist[enumerate,1]{
    label=(\roman*),
    ref=\roman*
}

\usepackage[sort,nocompress,noadjust]{cite}
\theoremstyle{plain}
\newtheorem{thm}{Theorem}
\newtheorem{theorem}[thm]{Theorem}
\newtheorem*{theorem*}{Theorem}

\newtheorem{lemma}[thm]{Lemma}

\newtheorem{cor}[thm]{Corollary}
\newtheorem{prop}[thm]{Proposition}

\theoremstyle{definition}
\newtheorem{definition}[thm]{Definition}
\newtheorem{remark}[thm]{Remark}

\newtheorem{example}[thm]{Example}

\numberwithin{equation}{section}
\numberwithin{thm}{section}
\def\N{{\mathbb N}}
\def\Z{{\mathbb Z}}
\def\R{{\mathbb R}}

\def\Q{{\mathbb Q}}

\usepackage{nmath}
\def\Rlex{\R^2_{\text{lex}}}

\allowdisplaybreaks

\begin{document}

\title[Rigidity of isometries on vector and normed lattices]{Rigidity of isometries on vector and normed lattices}

\author[Noel Murasko]{Noel Murasko}
\address{Department of Mathematical and Statistical Sciences, University of Alberta, Edmonton, AB, Canada, T6G 2G1}
\email{murasko@ualberta.ca}

\author[Tomasz Szczepanski]{Tomasz Szczepanski}
\address{Department of Mathematics, University of British Columbia, Vancouver, BC, Canada, V6T 1Z2}
\email{tomasz@math.ubc.ca}

\date{\today}
\subjclass[2020]{Primary 46A40; Secondary 46B04, 46B42, 39B52}
\keywords{Midpoint injectivity, isometries, vector lattices, normed lattices, functional equations, order preservation}

\begin{abstract}
We study rigidity phenomena for maps on vector and normed lattices arising from isometric, order-theoretic, and modulus-preserving conditions. We first investigate the role of strict convexity in Baker's nonsurjective version of the Mazur-Ulam theorem. In the process, we develop a theory of midpoint injective functions and investigate their connections with convexity and monotonicity. Using this framework, we prove a generalization of Baker's theorem.

As our first main application, we show that every modulus isometry, i.e., a map $T:X\to Y$, where $X$ is a sublattice of a vector lattice $Y$, satisfying 
$$|T(x)-T(y)|=|x-y|, \qquad x,y\in X,$$
is affine and has a disjointness-preserving linear part. When $X$ is an ideal, the linear part is moreover a bijection of $X$ onto itself and an involution. In the normed lattice setting, we also show that every positive norm isometry into a strictly convex normed lattice is not only affine, but has a linear part that is a lattice homomorphism.

Finally, we prove a nonlinear order-isometric rigidity theorem. After shifting to the origin, every order-preserving norm isometry between normed lattices whose codomain has a strictly monotone norm, preserves suprema and infima and is both disjointness preserving and disjointly additive. These conclusions do not force linearity, as shown by an explicit family of nonlinear isometries on~$L_1[0,1]$.
\end{abstract}

\maketitle



\section{Introduction}
\label{sec:introduction}
It is a recurring theme in functional analysis that maps defined by a metric or order-theoretic condition acquire additional algebraic structure that is not assumed \textit{a priori}. A classical example of this rigidity phenomenon is the Mazur-Ulam theorem, which states that a surjective isometry between real normed spaces must be affine. The original proof appeared in \cite{Mazur1932}, while shorter proofs can be
found in \cite{Vaisala2003,Nica2012}. The aim of this paper is to develop and study analogous rigidity phenomena in the setting of vector and normed lattices, where functional equations are induced by the lattice modulus or the order structure.

In general, the surjectivity assumption in the Mazur-Ulam theorem cannot be omitted; for example, $T\colon\R\to\ell_\infty^2$ defined by
$T(x):=(x,\sin x)$ is an isometry that is not affine. However, Baker \cite{Baker1971} showed that one can remove the requirement of surjectivity by replacing it with a geometric condition on the codomain.

There are several equivalent definitions of strictly convex normed spaces. For our purposes, it will be convenient to use the following: a normed space~$(Y,\norm{\cdot})$ is \emph{strictly convex} if it satisfies
\begin{equation}\label{eq:sc-norm-MI}
    \|x\|=\|y\|
=\left\|\frac{x+y}{2}\right\|
\quad\Longrightarrow\quad
x=y,
\qquad x,y\in Y.
\end{equation}
Baker's theorem may be stated as follows.
\begin{theorem}[{Baker \cite{Baker1971}}]
\label{thm:Baker}
If~$X$ and~$Y$ are normed spaces with~$Y$ strictly convex, and~$T:X\to Y$ is an isometry, then~$T$ is affine.
\end{theorem}
Baker's key insight in proving Theorem \ref{thm:Baker} was the relationship between a strictly convex normed space and midpoints, as exemplified by \eqref{eq:sc-norm-MI}. In this paper, we show that this relationship can be abstracted, which allows us to extend Baker's theorem to a much more general setting.

\subsection*{Midpoint injectivity} We begin by standardizing some terminology. Throughout, all vector spaces are assumed to be real and nontrivial. A map
$T\colon X\to Y$ between vector spaces is called \emph{$\Q$-linear} if it
satisfies
\[
T(px+qy)=pT(x)+qT(y),
\qquad p,q\in\Q,\quad x,y\in X.
\]
We say that~$T$ is \emph{$\Q$-affine} if the map~$U\deq T-T(0)$ is~$\Q$-linear. As usual, $T$ is called \emph{affine} if~$U$ is \emph{linear}. If~$(X,\norm{\cdot}_X)$ and~$(Y,\norm{\cdot}_Y)$ are normed spaces, we say that~$T:X\to Y$ is a (norm) \emph{isometry} if it satisfies
\begin{equation*}
\norm{T(x)-T(y)}_Y=\norm{x-y}_X,\qquad  x,y\in X.
\end{equation*}
We emphasize that in our work, maps (including isometries) are never assumed \textit{a priori} to be linear or affine.

Equation \eqref{eq:sc-norm-MI} motivates the following definition. We say that a map~$f\colon X\to Y$ between vector spaces is \emph{midpoint injective} if, for all~$x,y\in X$, we have
\begin{equation}
\label{eq:midpoint_injective}
f(x)=f(y)=f\left(\frac{x+y}{2}\right)
\quad\implies \quad
x=y.
\end{equation}
Thus, a normed space is strictly convex if and only if its norm, when viewed as a function into~$\R$, is midpoint injective. However, there are several other natural examples of midpoint injective functions. Every injective function is trivially midpoint injective, and, as we will see in Lemma \ref{lem:strictly-quasiconvex-implies-MI}, the same is true of every strictly quasiconvex function. Most importantly for our purposes, Lemma \ref{lem:modulus_is_midpoint_injective} shows that the modulus on any vector lattice is midpoint injective.

We show that midpoint injectivity is sufficient to
recover a rational version of Baker's affinity conclusion under a mild homogeneity assumption. We say that a map~$f:X\to Y$ between vector spaces is \emph{dyadically homogeneous} if it satisfies~$f(2^k x)=2^k f(x)$ for all~$x\in X$ and~$k\in\Z$.
\begin{lemma}
\label{lem:generalized-Mazur-Ulam-Baker}
Let~$X$, $Y$, and~$Z$ be vector spaces. Let
$f\colon Y\to Z$ and~$g\colon X\to Z$ be dyadically homogeneous
functions, and suppose additionally that~$f$ is midpoint injective. If
$T\colon X\to Y$ satisfies
\begin{equation}
\label{eq:generalized-isometry}
f\bigl(T(x)-T(y)\bigr)=g(x-y),
\qquad x,y\in X,
\end{equation}
then~$T$ is~$\Q$-affine.
\end{lemma}

In the special case in which~$Z=\R$, $X$ and~$Y$ are normed spaces,
and~$Y$ is strictly convex, we may take~$g(x):=\|x\|_X$ and~$f(y):=\|y\|_Y.$ Lemma~\ref{lem:generalized-Mazur-Ulam-Baker} then implies that every
norm isometry~$T\colon X\to Y$ is~$\Q$-affine. Since every norm isometry is continuous, every~$\Q$-affine isometry is affine, and Baker's theorem
follows.

\subsection*{Modulus isometries} Using Lemma \ref{lem:generalized-Mazur-Ulam-Baker}, we investigate rigidity phenomena in the setting of vector lattices. First, we recall some relevant terminology that will be used throughout this paper.

An \emph{ordered vector space} is a vector space~$X$ equipped with a partial order~$\leq$ that is compatible with the linear structure. If, in addition, every pair~$x,y\in X$ admits a supremum and an infimum, we call~$X$ a \emph{vector lattice}. We write
\[
x\vee y:=\sup\{x,y\},
\qquad
x\wedge y:=\inf\{x,y\}.
\]
The positive cone and modulus of~$X$ are then defined by
\[
X_+:=\{x\in X:x\geq 0\}
\quad\text{and}\quad
|x|:=x\vee(-x),
\]
respectively. We refer to~$|x|$, $x\wedge y$, and~$x\vee y$, together with~$x^+:=x\vee 0$ and~$x^-:=(-x)\vee 0$, as \emph{lattice operations}. We say two vectors~$x,y\in X$ are \emph{disjoint} and write~$x\perp y$ if~$|x|\wedge |y|=0$.

Suppose~$X$ is a vector lattice and that~$Y\subseteq X$. We say that~$Y$ is a \emph{sublattice} if it is a vector subspace of~$X$ closed under finite lattice operations. Moreover, a sublattice~$Y$ is an \emph{ideal} if whenever~$x\in X$ and~$y\in Y$ satisfy~$|x|\leq |y|$, it follows that~$x\in Y$.

Next, a map~$T:X\to Y$ between vector lattices~$X$ and~$Y$ is said to be
\begin{enumerate}
    \item \emph{positive} if~$T(X_+)\subseteq Y_+$;
    \item \emph{order-preserving} if~$T(x)\leq T(y)$ whenever~$x\leq y$;
    \item \emph{disjointness preserving} if~$T(x)\perp T(y)$ whenever~$x\perp y$;
    \item \emph{disjointly additive} (or orthogonally additive) if~$T(x+y)=T(x)+T(y)$ whenever~$x\perp y$.
\end{enumerate}
We say that~$T$ \emph{preserves all lattice operations} if it preserves the five lattice operations listed above. Finally, a linear map preserving all lattice operations is called a \emph{lattice homomorphism} and a bijective lattice homomorphism is called a \emph{lattice isomorphism}.
For further background on vector lattices, we refer the reader
to \cite[Chapter~1]{AliprantisBurkinshaw2006}.

The modulus map~$x\mapsto |x|$ on any vector lattice is dyadically homogeneous and, as we will see in
Lemma~\ref{lem:modulus_is_midpoint_injective}, it is also
midpoint injective. Expanding on these ideas allows us to prove the following statement, which may be regarded as a vector lattice analogue of the original Mazur-Ulam theorem.

\begin{theorem}
\label{thm:moduli_isometry}
Let~$Y$ be a vector lattice and let~$X\subseteq Y$ be a sublattice. If~$T\colon X\to Y$ satisfies
\begin{equation}\label{eq:moduli_isometry}
|T(x)-T(y)|=|x-y|, \qquad x,y\in X,
\end{equation}
then~$T-T(0)$ is a disjointness preserving linear map. Furthermore, if~$X$ is an ideal, then~$\operatorname{Range}(T-T(0))=X$ and
$(T-T(0))^2=\operatorname{Id}_X$.
\end{theorem}
Interpreting $|x-y|$ as a vector-valued
distance, we refer to maps satisfying \eqref{eq:moduli_isometry} as \emph{modulus isometries}.

The affinity conclusion in Theorem \ref{thm:moduli_isometry} can be seen as a consequence of a more general principle. In Theorem \ref{thm:general_equation_is_Q_linear} we show that if the expression
\begin{equation*}
|T(x)-T(y)|
\end{equation*}
depends only on~$x-y$ for every~$x,y\in X$, then~$T$ must be~$\Q$-affine. The key observation is that if the function~$f$ in Lemma \ref{lem:generalized-Mazur-Ulam-Baker} is the modulus of a vector lattice, then the corresponding function~$g$ is \emph{necessarily} dyadically homogeneous. As a further consequence of Theorem \ref{thm:moduli_isometry}, we obtain in Corollary \ref{cor:T|x-y|_is_lattice_homomorphism} a functional-equation characterization of lattice homomorphisms.

\subsection*{Positive and order-preserving isometries} We next turn to positive isometries between normed lattices. A norm~$\norm{\cdot}$ on a vector lattice~$X$ is called a \emph{lattice norm} if it satisfies
$$
|x|\leq |y|
\quad\Longrightarrow\quad
\|x\|\leq\|y\|,
\qquad x,y\in X.
$$
We call~$(X,\norm{\cdot})$ a \emph{normed lattice}. A lattice norm is called \emph{strictly monotone} if it satisfies~$\norm{x}<\norm{y}$ whenever~$|x|<|y|$.

To motivate our investigation, we recall a result of Abramovich. If~$X$ and~$Y$ are normed lattices and~$T:X\to Y$ is a positive and surjective isometry, then~$T-T(0)$ is a lattice isomorphism; see \cite[Theorem 1]{Abramovich1988}. We note that in \cite{Abramovich1988}, Abramovich assumes that all isometries are linear, but this assumption is not needed for his theorem. Indeed, since~$T$ is surjective, the Mazur-Ulam theorem guarantees that~$T$ must be affine. Moreover, a positive affine map into a normed lattice has a positive linear part~$T-T(0)$. Thus, Abramovich's theorem provides a lattice-theoretic counterpart of the Mazur-Ulam theorem for positive isometries.

Abramovich also noted that, even under the assumption of linearity, surjectivity is generally necessary to conclude that the map is a lattice homomorphism. However, using properties of midpoint injective functions, we show that one may replace the surjectivity requirement with strict convexity of the codomain; this result may be seen as a Baker-type analogue of Abramovich's theorem.

\begin{prop}\label{prop:Abramovich-Baker-version}
If~$X$ and~$Y$ are normed lattices with the norm on~$Y$ strictly convex, and~$T:X\to Y$ is a positive isometry, then~$T-T(0)$ is a lattice homomorphism.
\end{prop}

The situation changes, however, when the codomain is not strictly convex. The following result, which Abramovich attributed to Veksler, shows that surjectivity can still be removed if the norm on the codomain is strictly monotone, provided affinity is assumed.

\begin{theorem}[{Veksler \cite[Corollary 6]{Abramovich1988}}]\label{thm:Veksler}
If~$X$ and~$Y$ are normed lattices with the norm on~$Y$ strictly monotone, and~$T:X\to Y$ is a positive and affine isometry, then~$T-T(0)$ is disjointness preserving. Consequently, $T-T(0)$ is a lattice homomorphism.
\end{theorem}

In contrast to Abramovich's result, Example \ref{ex:Veksler-linear} shows that one cannot generally remove the assumption that~$T$ is affine in Theorem \ref{thm:Veksler}. Motivated by this, we investigate what can be said about isometries into strictly monotone normed lattices, when the assumption on positivity and affinity is replaced with order preservation.

A finite-dimensional result in this direction was obtained by
Lemmens and van Gaans \cite[Proposition~4.1]{LemmensVanGaans2003}. They showed that if~$Y\subseteq\R^n$ is a lattice containing~$0$ and~$f:Y\to Y$ is an order-preserving isometry satisfying~$f(0)=0$, then~$f$ preserves suprema and infima whenever the ambient norm is monotone. Their result concerns self-maps of finite-dimensional lattice subsets. In contrast, our next theorem applies to maps between arbitrary normed lattices, at the expense of assuming that the norm on the codomain is strictly monotone.

\begin{theorem}\label{thm:nonlinear-Veksler}
Let~$X$ and~$Y$ be normed lattices, with the norm on~$Y$ strictly monotone, and let~$S:X\to Y$ be an order-preserving isometry. Then \mbox{$T\deq S-S(0)$} preserves suprema and infima. Moreover, $T$ is disjointness preserving and disjointly additive. Furthermore, $T$ preserves all lattice operations if and only if it is odd.
\end{theorem}

Since every positive affine map into a normed lattice is order-preserving, Theorem \ref{thm:nonlinear-Veksler} can be used to recover Theorem \ref{thm:Veksler}. Importantly, the hypotheses of Theorem \ref{thm:nonlinear-Veksler} do not force~$S$ to be affine. In Example \ref{ex:nonlinear-Veksler} we construct an explicit family of nonlinear order-preserving isometries on~$L_1[0,1]$.

Theorem \ref{thm:nonlinear-Veksler} connects order-preserving isometries with several established classes of nonlinear operators. Since the normalized map~$T=S-S(0)$ is order-preserving, it is order bounded. The theorem shows that~$T$ is disjointly additive, and hence~$T$ is an \textit{abstract Urysohn operator}. Such operators have been extensively studied; see, for example, \cite{MazonSegura1990,AbasovPliev2018,Feldman2019,Pliev2021} and the references therein.

The simultaneous disjointness-preserving conclusion is also related to the nonlinear operators studied in \cite{Feldman2017,AbasovPliev2018}. In the odd case, $T$ additionally preserves all lattice operations, connecting the theorem with nonlinear lattice-preserving maps such as those studied by Feldman \cite{Feldman2013}. More broadly, nonlinear order isomorphisms and the extent to which they retain the structure of function spaces and Banach lattices have been investigated by Leung and Tang \cite{LeungTang2016, LeungTangPersistence2016}.

The structure of the paper is as follows. In Section~\ref{sec:midpoint_injectivity}, we develop a theory of midpoint
injective functions and prove Lemma~\ref{lem:generalized-Mazur-Ulam-Baker}. We also explore the connection between midpoint injectivity and various notions of convexity.

In Section~\ref{sec:mazur_ulam_baker} we study consequences of the generalized distance identity~\eqref{eq:generalized-isometry}. In particular, we prove Theorem~\ref{thm:moduli_isometry}.

Finally, in Section \ref{sec:isometries_in_normed_lattices} we investigate the relationship between midpoint injectivity and monotonicity, and prove Proposition \ref{prop:Abramovich-Baker-version} and Theorem \ref{thm:nonlinear-Veksler}.

\section{Midpoint injectivity}
\label{sec:midpoint_injectivity}

In this section, we develop tools for studying midpoint injective
functions, i.e., functions satisfying \eqref{eq:midpoint_injective}. These tools will be used in the proofs of the main results. As noted in the introduction, midpoint injective functions arise in the setting of vector and normed lattices, as well as in general normed spaces. We therefore investigate midpoint injectivity as a property of interest in its own right. In particular, in Subsection \ref{sub:midpoint_injectivity_and_convexity}, we study its connection with convexity and related geometric properties.

\subsection{A generalized Baker theorem}

We begin with the following proposition, which collects several useful equivalent formulations of midpoint injectivity.

\begin{prop}
\label{prop:MI-equiv}
Let~$X$ and~$Y$ be vector spaces, and let~$f\colon X\to Y$ be a
function. Then the following statements are equivalent:
\begin{enumerate}
\item\label{item:MI-equiv-MI}~$f$ is midpoint injective;

\item\label{item:MI-equiv-UMP} for all~$u,v,m\in X$,
$$
f(u-m)=f(m-v)=f\left(\frac{u-v}{2}\right)
\quad\Longrightarrow\quad
m=\frac{u+v}{2};
$$

\item\label{item:MI-equiv-MI-convenient} for all~$w,z\in X$,
$$
f(w)=f(z)=f(2w-z)
\quad\Longrightarrow\quad
w=z.
$$
\end{enumerate}
\end{prop}

\begin{proof}
\noindent
\textit{(\ref*{item:MI-equiv-UMP})$\implies$(\ref*{item:MI-equiv-MI})}. Suppose that~$f(x)=f(y) = f\left(\frac{x+y}{2}\right)$ for some~$x,y\in X$. Then
$$f(x-0)=f(0-(-y))=f\left(\frac{x-(-y)}{2}\right).$$
Hence~$0=\frac{x-y}{2}$, and therefore~$x=y$.

\smallskip

\noindent
\textit{(\ref*{item:MI-equiv-MI})$\implies$(\ref*{item:MI-equiv-MI-convenient})}. Suppose that~$f(w)=f(z) = f(2w-z)$ for some~$w,z\in X$. Letting~$x\deq 2w-z$ and~$y\deq z$, we have~$f(x)=f(y)=f\left(\frac{x+y}{2}\right)$. Midpoint injectivity yields~$x=y$, and hence~$w=z$.

\smallskip

\noindent
\textit{(\ref*{item:MI-equiv-MI-convenient})$\implies$(\ref*{item:MI-equiv-UMP})}. Suppose that~$f(u-m)=f(m-v) = f\left(\frac{u-v}{2}\right)$ for some~$u,v,m\in X$. Letting~$w\deq \frac{u-v}{2}$ and~$z\deq m-v$ implies that~$f(w)=f(z)=f(2w-z)$. Condition (\ref*{item:MI-equiv-MI-convenient}) gives~$w=z$, and therefore~$m=\frac{u+v}{2}$.
\end{proof}

We note that in \cite[Lemma~2]{Baker1971}, Baker proved that every norm on a strictly convex normed space satisfies (\ref*{item:MI-equiv-UMP}) of Proposition \ref{prop:MI-equiv}. Indeed, it is precisely this property that is needed in the proof of Baker's theorem. The final necessary ingredient is the standard characterization of~$\Q$-affine maps in terms of midpoint preservation.

\begin{lemma}
\label{lem:midpoint-implies-Qaffine}
Let~$X$ and~$Y$ be vector spaces. A map~$T\colon X\to Y$ is
$\Q$-affine if and only if
\begin{equation*}
T\left(\frac{x+y}{2}\right)
=
\frac{T(x)+T(y)}{2},
\qquad x,y\in X.
\end{equation*}
\end{lemma}

With this, we are ready to prove Lemma \ref{lem:generalized-Mazur-Ulam-Baker}.
\begin{proof}[Proof of Lemma \ref{lem:generalized-Mazur-Ulam-Baker}.]
We essentially follow Baker's proof in \cite{Baker1971}. Let~$g:X\to Z$ and~$f:Y\to Z$ be dyadically homogeneous functions, with~$f$ midpoint injective, and satisfying
$$g(x-y)=f(T(x)-T(y)),\qquad x,y\in X.$$
Fix~$x,y\in X$ and define~$m\deq \frac{x+y}{2}$. Using dyadic homogeneity of~$g$, we get
\begin{equation}\label{eq:mid1}
f(T(x)-T(m))=g\left(x-m\right)=\frac{1}{2}g\left(x-y\right),
\end{equation}
and
\begin{equation}\label{eq:mid2}
f(T(m)-T(y))=g\left(m-y\right)=\frac{1}{2}g\left(x-y\right).
\end{equation}
By combining \eqref{eq:mid1} and \eqref{eq:mid2}, together with dyadic homogeneity of~$f$, we obtain
\begin{align*}
    f(T(x)-T(m))&=f(T(m)-T(y))=\frac{1}{2}g\left(x-y\right)\\
    &=\frac{1}{2}f\left(T(x)-T(y)\right)=f\left(\frac{T(x)-T(y)}{2}\right).
\end{align*}
Since~$f$ is midpoint injective, by Proposition \ref{prop:MI-equiv}, $T(m)=\frac{T(x)+T(y)}{2}$. Thus~$T$ preserves midpoints, and so by Lemma \ref{lem:midpoint-implies-Qaffine}, $T$ is~$\Q$-affine.
\end{proof}

\subsection{Midpoint injectivity and convexity}\label{sub:midpoint_injectivity_and_convexity}

The connection between midpoint injectivity and strict convexity of normed spaces naturally raises the question of how midpoint injectivity interacts with more general notions of convexity. Thus, we now turn our attention to the relationship between midpoint injectivity and convexity of vector-lattice-valued functions.

\begin{definition}
Let~$X$ be a vector space, let~$Y$ be an ordered vector space, and let
$f\colon X\to Y$.

\begin{enumerate}
\item The function~$f$ is called \textit{convex} if
\begin{equation}\label{eq:convex_definition}
f\bigl((1-\lambda)x+\lambda y\bigr)
\leq
(1-\lambda)f(x)+\lambda f(y),
\end{equation}
for every~$x,y\in X$ with~$x\neq y$ and every~$\lambda\in(0,1)$. It is called \textit{strictly convex} if the inequality in \eqref{eq:convex_definition} is strict.

\item If~$Y$ is a vector lattice, then~$f$ is called
\textit{quasiconvex} if
\begin{equation}\label{eq:quasiconvex_definition}
f\bigl((1-\lambda)x+\lambda y\bigr)
\leq
f(x)\vee f(y)
\end{equation}
for every~$x,y\in X$ with~$x\neq y$ and every~$\lambda\in(0,1)$. It is called \textit{strictly quasiconvex} if the inequality in \eqref{eq:quasiconvex_definition} is strict.
\end{enumerate}
\end{definition}
We note that convex and quasiconvex functions with values in vector lattices have been studied in the theories of convex operators and vector optimization; see, for example, the monograph \cite{KusraevKutateladze1995} as well as \cite{MartinezLegaz2012}. Our treatment of these notions, however, is self-contained.

When the codomain is a vector lattice, it is easy to see that every strictly convex function is strictly quasiconvex, and every convex function is quasiconvex. The following gives an example of a convex map on a vector lattice that fails to be strictly quasiconvex. We will refer to it later.

\begin{example}\label{ex:positive-part-function}
    Let~$X$ be a vector lattice. Consider the function~$x\mapsto x^+$ on~$X$. Recalling the identity~$x^+=\frac{1}{2}(|x|+x)$, we have
  \begin{align*}
  (\lambda x+(1-\lambda)y)^+&=\frac{1}{2}\left[|\lambda x+(1-\lambda)y|+\lambda x+(1-\lambda)y\right]\\
  &\leq \lambda \cdot \frac{1}{2}\left(|x|+x\right)+(1-\lambda) \cdot \frac{1}{2}\left(|y|+y\right)\\
  &=\lambda x^++(1-\lambda)y^+,
  \end{align*}
  so~$x\mapsto x^+$ is convex, and hence quasiconvex. On the other hand, let~$x,y<0$ be distinct. Then~$x^+=y^+=0$, and so for any~$\lambda \in (0,1)$ we have
  $$(\lambda x+(1-\lambda) y)^+=0=x^+\vee y^+.$$
  Thus, $x\mapsto x^+$ is not strictly quasiconvex.
\end{example}

Example \ref{ex:positive-part-function} also shows that quasiconvexity alone does not imply midpoint injectivity. Indeed, it is easy to observe that the map~$x\mapsto x^+$ is not midpoint injective; taking any~$x,y<0$ with~$x\neq y$ gives~$x^+=y^+=\big(\frac{x+y}{2}\big)^+=0$. The next lemma shows that under a stronger assumption of strict quasiconvexity, midpoint injectivity always follows.
\begin{lemma}\label{lem:strictly-quasiconvex-implies-MI}
    Let~$X$ be a vector space and let~$Y$ be a vector lattice. If~$f\colon X\to Y$ is strictly quasiconvex, then~$f$ is midpoint injective.
\end{lemma}
\begin{proof}
    Suppose that for some~$x,y\in X$ we have
        $$f(x)=f(y)=f\left(\frac{x+y}{2}\right).$$
    This gives~$f\left(\frac{x+y}{2}\right)=f(x)\vee f(y)$. If~$x\neq y$, then strict quasiconvexity implies that~$f\left(\frac{x+y}{2}\right)<f(x)\vee f(y)$, giving a contradiction. Hence~$x=y$, so~$f$ is midpoint injective.
\end{proof}

It is natural to ask whether the converse to Lemma \ref{lem:strictly-quasiconvex-implies-MI} holds. It is not true in general even on~$\R$, as there exist injective (hence midpoint injective) non-quasiconvex functions. Nevertheless, the following statement shows that, within the class of convex functions, the concepts of midpoint injectivity and strict quasiconvexity coincide.

\begin{prop}
\label{prop:MI-convex}
Let~$X$ be a vector space, let~$Y$ be a vector lattice, and let
$f\colon X\to Y$ be convex. Then the following statements are
equivalent:
\begin{enumerate}
\item\label{item:MI-convex-MI}~$f$ is midpoint injective;

\item\label{item:MI-convex-SQC}~$f$ is strictly quasiconvex;

\item\label{item:MI-convex-CI} for all~$x,y\in X$ and every~$\lambda\in(0,1)$,
$$
f(x)=f(y)
=
f\bigl((1-\lambda)x+\lambda y\bigr)
\quad\Longrightarrow\quad
x=y;
$$

\item\label{item:MI-convex-zero} for all~$x,y\in X$ and all distinct nonzero~$a,b\in\R$,
$$
f(x+ay)=f(x)=f(x+by)
\quad\Longrightarrow\quad
y=0.
$$
\end{enumerate}
\end{prop}

\begin{proof}
\noindent
\textit{(\ref*{item:MI-convex-SQC})$\implies$(\ref*{item:MI-convex-MI})}. Follows from Lemma \ref{lem:strictly-quasiconvex-implies-MI}.

\smallskip
\noindent
\textit{(\ref*{item:MI-convex-MI})$\implies$(\ref*{item:MI-convex-CI})}. Assume that for some~$x,y\in X$ and~$\lambda \in (0,1)$ we have
$$f(x)=f(y)=f(\lambda x+(1-\lambda)y).$$
Without loss of generality, assume that~$\lambda \leq \frac{1}{2}$; otherwise interchange~$x$ and~$y$ and replace~$\lambda$ with~$1-\lambda$. Let~$m=\frac{x+y}{2}$. As~$f$ is convex, we have $$f(m)\leq \frac{f(x)+f(y)}{2}=f(y).$$
Note that
\begin{align*}
\lambda x+(1-\lambda)y&=\lambda x+\lambda y - \lambda y+(1-\lambda)y=2\lambda m
+(1-2\lambda)y.
\end{align*}
Because~$\lambda \leq \frac{1}{2}$, this shows that the former is a convex combination of~$m$ and~$y$. Using convexity of~$f$ one more time, we obtain
\begin{align*}
f(y)&=f(\lambda x+(1-\lambda)y)=f(2\lambda m +(1-2\lambda)y)\\
&\leq 2\lambda f(m)+(1-2\lambda)f(y)\leq 2\lambda f(y)+(1-2\lambda)f(y)\\
&=f(y).
\end{align*}
Thus, we must have~$f(m)=f(y)$, which implies
$$f(x)=f(y)=f\left(\frac{x+y}{2}\right).$$
Because~$f$ is midpoint injective, we then have~$x=y$.

\smallskip
\noindent
\textit{(\ref*{item:MI-convex-CI})$\implies$(\ref*{item:MI-convex-SQC})}. Let~$x,y\in X$ with~$x\neq y$ and let~$\lambda \in (0,1)$. If~$f(x)\neq f(y)$, then by convexity we have
\begin{align*}
f(\lambda x +(1-\lambda) y)&\leq \lambda f(x)+(1-\lambda)f(y)\\
&< \lambda[f(x)\vee f(y)]+(1-\lambda) [f(x)\vee f(y)]\\
&=f(x)\vee f(y).
\end{align*}
On the other hand, if~$f(x)=f(y)$, then we have
\begin{equation}\label{eq:fx-equals-fy}
f(\lambda x +(1-\lambda) y)\leq \lambda f(x)+(1-\lambda) f(y)=f(x)\vee f(y).
\end{equation}
It remains to show that the inequality in \eqref{eq:fx-equals-fy} is strict. To this end, assume that \eqref{eq:fx-equals-fy} is an equality. We would then have
$$f(x)=f(y)=f(x)\vee f(y)=f(\lambda x +(1-\lambda) y),$$
which by (\ref*{item:MI-convex-CI}) would imply that~$x=y$, giving a contradiction.

\smallskip
\noindent
\textit{(\ref*{item:MI-convex-CI})$\implies$(\ref*{item:MI-convex-zero})}. Assume that for some~$x,y\in X$ and~$a,b\in \R\setminus \{0\}$ with~$a\neq b$
\begin{equation}\label{eq:convex-zero}
  f(x+ay)=f(x)=f(x+by).
\end{equation}
By way of contradiction, suppose that~$y\neq 0$. Since~$0,a,b$ are three distinct real numbers, one lies strictly between the other two. Hence, one of the three points~$x+ay$, $x$, and~$x+by$ is a nontrivial convex combination of the other two. Combining this with \eqref{eq:convex-zero}, we see that assumption (\ref*{item:MI-convex-CI}) implies that the two outer points coincide. Since their corresponding parameters are distinct, this forces~$y=0$, a contradiction.

\smallskip
\noindent
\textit{(\ref*{item:MI-convex-zero})$\implies$(\ref*{item:MI-convex-MI})}. Assume that for some~$x,y\in X$ we have
$$f(x)=f(y)=f\left(2x-y\right).$$
By Proposition \ref{prop:MI-equiv} (\ref*{item:MI-equiv-MI-convenient}), it suffices to prove that~$x=y$.

Let~$u\deq 2x-y$ and~$v\deq y-x$. Then we have~$x=u+v$, $y=u+2v$, and
$$2x-y=2u+2v-u-2v=u.$$
Thus we have
$$f(u+v)=f(u+2v)=f(u).$$
Applying (\ref*{item:MI-convex-zero}) with~$a=1$ and~$b=2$
gives~$v=0$, and hence~$x=y$.
\end{proof}

We conclude this section by investigating the convexity assumption in Proposition \ref{prop:MI-convex}. The following proposition shows that convexity may be replaced by quasiconvexity for functions taking values in a totally-ordered vector lattice (in which every two vectors are comparable). In particular, for real-valued quasiconvex functions, midpoint injectivity is equivalent to strict quasiconvexity.

\begin{prop}\label{prop:MI-totally-ordered}
Let~$X$ be a vector space, let~$Y$ be a totally ordered vector lattice, and let~$f:X\to Y$ be quasiconvex. Then~$f$ is midpoint injective if and only if~$f$ is strictly quasiconvex.
\end{prop}

\begin{proof}
    If~$f$ is strictly quasiconvex, Lemma \ref{lem:strictly-quasiconvex-implies-MI} implies that it is midpoint injective. Thus it suffices to prove the other implication.

    We begin with the following preliminary observation. Let~$f:X\to Y$ be quasiconvex and midpoint injective and fix~$x,y\in X$ with~$x\neq y$. Define
    $$\gamma:[0,1]\to X, \qquad \gamma(t)=tx+(1-t)y.$$
    Since~$x\neq y$, $\gamma$ is injective. Put~$g:=f\circ \gamma$. We note that since~$f$ is quasiconvex and~$\gamma$ is affine, it follows that~$g$ is also quasiconvex.

    We claim the following: for a fixed~$0\leq a<b<c\leq 1$, neither~$g(a)=g(b)>g(c)$ nor~$g(a)<g(b)=g(c)$ can occur. Indeed, suppose that~$g(a)=g(b)>g(c)$ and set~$m=\frac{a+b}{2}$. By quasiconvexity we get~$g(m)\leq g(a)$. Since~$b$ lies between~$m$ and~$c$, quasiconvexity also gives~$g(b)\leq g(m)\vee g(c)$. Using the fact that~$Y$ is totally ordered and~$g(c)<g(b)$, it follows that~$g(b)\leq g(m)$, as otherwise we would have~$g(m)\vee g(c)<g(b)$. Hence~$g(a)=g(m)=g(b).$ Since~$\gamma(m)=\frac{\gamma(a)+\gamma(b)}{2}$ and since~$f$ is midpoint injective, this gives~$\gamma(a) = \gamma(b)$. From injectivity of~$\gamma$, this gives~$a=b$, which contradicts the assumption. The second case is analogous.

    We are now ready to prove the proposition. By way of contradiction, suppose that~$f$ is not strictly quasiconvex. Then there exist~$\lambda\in (0,1)$ and~$x,y\in X$ such that~$x\neq y$ and
    $$f(\lambda x+(1-\lambda)y) = f(x)\vee f(y).$$
    Using the function~$g=f\circ \gamma$ defined above together with the claim, we will now conclude the proof. As~$Y$ is totally ordered, there are three cases to consider.

    If~$f(x)<f(y)$, then~$g(0)=g(\lambda)>g(1)$, contradicting the claim. Similarly, if~$f(x)>f(y)$, then~$g(0)<g(\lambda)=g(1)$, which also contradicts the claim.

    Suppose now~$f(x)=f(y)$. Since~$x\neq y$, quasiconvexity and midpoint injectivity imply that
    $$g\left(\frac12\right)<g(0)=g(1).$$
    If~$\lambda=\frac12$, this already contradicts~$g(\lambda)=g(0)$. If~$\lambda<\frac12$ then
    $$g(0)=g(\lambda)>g\left(\frac12\right),$$
    while if~$\lambda>\frac12$, then
    $$g\left(\frac12\right)<g(\lambda)=g(1).$$
    Either case contradicts the claim. Therefore~$f$ is strictly quasiconvex.
\end{proof}

It is natural to ask whether one might always be able to replace convexity with quasiconvexity in Proposition \ref{prop:MI-convex}. Our final example of this section shows that this is not the case for functions taking values in a partially ordered vector lattice. In particular, this demonstrates that midpoint injectivity is more subtle for vector-lattice-valued functions, where phenomena arise that have no counterpart in the classical real-valued setting.

\begin{example}
Consider~$\R^2$ with the coordinate-wise order and let~$f:\R\to \R^2$ be defined by~$f(x)=-(x^+,x^-)$. As~$f$ is injective, it is midpoint injective. Moreover, the functions~$x\mapsto -(x^+)$ and~$x\mapsto -(x^-)$ are both quasiconvex, since they are non-increasing and non-decreasing, respectively. Hence, with respect to the coordinate-wise order on~$\R^2$, the map~$f$ must be quasiconvex as well.

Now consider~$x=1$ and~$y=-1$. Then~$f(x)=(-1,0)$ and~$f(y)=(0,-1)$, so~$f(x)\vee f(y)=(0,0)$. We also have~$f(\frac{x+y}{2})=f(0)=(0,0)$, so~$f$ is not strictly quasiconvex.
\end{example}

\section{Modulus equations in vector lattices}
\label{sec:mazur_ulam_baker}

We now apply Lemma \ref{lem:generalized-Mazur-Ulam-Baker} to the lattice modulus. Recall that, in a vector lattice~$X$, two vectors~$x,y\in X$ are called disjoint, written~$x\perp y$, if~$|x|\wedge |y|=0$. Disjointness plays an important role in this section; hence we recall some of the known basic properties that will be used throughout.

We begin with the standard identity (see \cite[Theorem 1.7 (5)]{AliprantisBurkinshaw2006}) that for any~$x,y\in X$ we have
\begin{equation}\label{eq:disjointness-fact}
  |x+y|=|x-y|\iff x\perp y.
\end{equation}
Combining this identity with \cite[Theorem 1.7 (7)]{AliprantisBurkinshaw2006} and \cite[Theorem 1.7 (4)]{AliprantisBurkinshaw2006}, we obtain
\begin{equation}\label{eq:disjointness-fact2}
|x+y|=|x|+|y|=|x|\vee |y|,
\end{equation}
whenever~$x\perp y$.

Before we turn our attention to applications of Lemma \ref{lem:generalized-Mazur-Ulam-Baker} to general vector lattices, we first establish the following key fact.
\begin{lemma}
\label{lem:modulus_is_midpoint_injective}
Let~$X$ be a vector lattice. Then the modulus map
$$
x\longmapsto |x|
$$
is midpoint injective.
\end{lemma}
\begin{proof}
Let~$x,y\in X$ be such that~$|x|=|y|=\frac{|x+y|}{2}$.
We make use of the following standard identity (see \cite[Theorem 1.7 (6)]{AliprantisBurkinshaw2006})
$$|x+y|\wedge |x-y|=\left||x|-|y|\right|.$$
As~$|x|=|y|$, this implies that~$x+y\perp x-y$. Thus \eqref{eq:disjointness-fact2} gives us
$$2|x|+|x-y|=|x+y|+|x-y|=|x+y+x-y|=2|x|,$$
so~$|x-y|=0$, and hence~$x=y$.
\end{proof}

\subsection{Translation-invariant modulus equations}

Since the modulus is midpoint injective and dyadically homogeneous, it provides another fundamental class of functions to which Lemma \ref{lem:generalized-Mazur-Ulam-Baker} applies. Our next goal is to investigate a general functional equation from which Theorem \ref{thm:moduli_isometry} will follow.

\begin{theorem}
\label{thm:general_equation_is_Q_linear}
Let~$X$ and~$Y$ be vector lattices and let~$T\colon X\to Y$ be a map.
Suppose that there exists a function~$g\colon X\to Y_+$ such that
\[
|T(x)-T(y)|=g(x-y),
\qquad x,y\in X.
\]
Then~$T$ is~$\Q$-affine. Furthermore, if
$$g(|x|)=g(x), \qquad x\in X,$$
then~$T-T(0)$ is disjointness preserving.
\end{theorem}

Before proving the theorem, we reduce it to a normalized equation. Suppose~$T:X\to Y$ is a map satisfying
$$|T(x)-T(y)|=g(x-y), \qquad x,y\in X$$
for some~$g:X\to Y_+$. Consider the map~$U:=T-T(0)$. Then~$U(0)=0$ and
\begin{align*}
|U(x)-U(y)|&=|T(x)-T(y)|=g(x-y)\\
&=g(x-y-0) = |T(x-y)-T(0)| \\
&= |U(x-y)|.
\end{align*}
Moreover, $g(x)=|U(x)|$. Thus, in order to study maps~$T:X\to Y$ satisfying~$|T(x)-T(y)|=g(x-y)$ for all~$x,y\in X$, it suffices to consider the normalized equation
$$U(0)=0,\qquad |U(x)-U(y)|=|U(x-y)|.$$

The normalization argument is not limited to equations involving the lattice modulus. For example, maps~$f:X\to Y$ between normed spaces satisfying the equation
$$
\|f(x)-f(y)\|=\|f(x-y)\|, \qquad x,y\in X,
$$
and closely related norm-valued functional equations have been studied by several authors; see, for example,
\cite{Ger1993,Skof1993, Sikorska2005,Dong2015} and the references therein.

\begin{proof}[Proof of Theorem \ref{thm:general_equation_is_Q_linear}]
By the preceding normalization, we may assume that~$T:X\to Y$ is a map satisfying~$T(0)=0$ and
\begin{equation}
\label{eq:general_equation}
|T(x)-T(y)|=|T(x-y)|,
\qquad x,y\in X.
\end{equation}

To prove that~$T$ is~$\Q$-linear, it suffices to show that~$|T(\cdot)|$ is dyadically homogeneous. Indeed, since the modulus map is midpoint injective by Lemma \ref{lem:modulus_is_midpoint_injective}, the claim will then follow from Lemma \ref{lem:generalized-Mazur-Ulam-Baker} applied with~$f=|\cdot|$ and~$g=|T(\cdot)|$. In fact, we prove the stronger identity: for every integer~$n\geq 2$ and~$x\in X$, the map~$T$ satisfies
\begin{equation}
\label{eq:integer_differences}
T\left(\frac{nx}{2}\right)
-
T\left(\frac{(n-2)x}{2}\right)
=
T(x).
\end{equation}

To prove \eqref{eq:integer_differences}, we fix~$x\in X$ and define~$r\deq T\left(\frac{x}{2}\right)-\frac{1}{2}T(x)$. By \eqref{eq:general_equation}, we have
\begin{align*}
    \left|\frac{1}{2}T(x) -r\right|&=\left|T(x) -T\left(\frac{x}{2}\right)\right|=\left|T\left(\frac{x}{2}\right)\right|=\left|\frac{1}{2}T(x)+T\left(\frac{x}{2}\right)-\frac{1}{2}T(x)\right|\\
    &=\left|\frac{1}{2}T(x)+r\right|.
\end{align*}
Thus identity \eqref{eq:disjointness-fact} implies that~$r\perp \frac{1}{2}T(x)$.

We now define
$$a_n\deq T\left(\frac{nx}{2}\right)- T\left(\frac{(n-2)x}{2}\right),\quad  n\in\N, \quad n\geq 2.$$
We use induction to prove that~$a_n=T(x)$ for all~$n\geq 2$. Clearly, $a_2=T(x)$. Suppose inductively that~$a_j=T(x)$ for all~$j=2,\ldots, n-1$. Define
$$b_n\deq T\left(\frac{nx}{2}\right)- T\left(\frac{(n-1)x}{2}\right),\quad  n\in\N, \quad n\geq 1.$$
Note that~$b_1=T\left(\frac{x}{2}\right)=\frac{1}{2}T(x)+r$. Since~$b_j+b_{j-1}=a_j$, we get~$b_j+b_{j-1}=T(x)$ for all~$j=2,\dots n-1$. Thus
$$b_j - \frac{1}{2}T(x) = -(b_{j-1} - \frac{1}{2}T(x)), \quad j=2,\dots,n-1.$$
Applying this formula recursively, we get
$$b_{n-1} -\frac{1}{2}T(x) = (-1)^{n-2}(b_1 - \frac{1}{2}T(x)) = (-1)^{n-2}r.$$
Since~$b_n=a_n - b_{n-1}$, this gives
$$b_n = a_n - \frac{1}{2}T(x) + (-1)^{n-1}r.$$
Then by \eqref{eq:general_equation} we have that~$|a_n|=|T(x)|$. As~$r\perp T(x)$, we conclude that~$a_n \perp r$, and in particular also~$r\perp (a_n - \frac{1}{2}T(x))$. Using this, we obtain
\begin{equation}\label{eq:on_the_one_hand}
|b_n|=\left|a_n - \frac{1}{2}T(x)\right| + |r|.
\end{equation}
On the other hand, \eqref{eq:general_equation} also gives us
\begin{equation}\label{eq:on_the_other_hand}
|b_n|=\left|T\left(\frac{x}{2}\right)\right|=\left|\frac{1}{2}T(x)+r\right|=\left|\frac{1}{2}T(x)\right| + |r|.
\end{equation}
Combining \eqref{eq:on_the_one_hand} and \eqref{eq:on_the_other_hand} then gives
$$\left|a_n - \frac{1}{2}T(x)\right|=\frac{1}{2}\left|T(x)\right|=\frac{1}{2}|a_n|.$$
Thus, setting~$w=\frac{1}{2}a_n$ and~$z=\frac{1}{2}T(x)$ gives
$$|2w-z|=|z|=|w|.$$
Therefore, applying Proposition \ref{prop:MI-equiv} (\ref*{item:MI-equiv-MI-convenient}) to the modulus map yields~$w=z$. Hence~$a_n=T(x)$, which completes the proof of \eqref{eq:integer_differences}.

Taking~$n=4$ in \eqref{eq:integer_differences} gives~$T(2x)=2T(x)$. Consequently, $|T(\cdot)|$ is dyadically homogeneous. Lemma \ref{lem:generalized-Mazur-Ulam-Baker} therefore implies that~$T$ is~$\Q$-linear.

To prove the second part of the claim, we assume that~$g(|x|)=g(x)$ for all~$x\in X$. In particular, this implies that
$|T(x)|=g(x) = g(|x|) = |T(|x|)|.$
Then, if~$x,y\in X$ are disjoint, identity \eqref{eq:disjointness-fact} gives us
$$|T(x)-T(y)|=|T(|x-y|)|=|T(|x+y|)|=|T(x)+T(y)|,$$
so~$T(x)\perp T(y)$ as well.
\end{proof}
\begin{remark}\label{re:Q_linear_not_linear}
Maps satisfying \eqref{eq:general_equation} need not be linear. Indeed, if~$T\colon\R\to\R$ is any~$\Q$-linear map, then it clearly satisfies~\eqref{eq:general_equation}. In particular, we may take~$T$ to be a~$\Q$-linear function that is not linear over
$\R$. Such a function may be constructed by choosing a Hamel basis of
$\R$ over~$\Q$.
\end{remark}

We now investigate additional assumptions under which a~$\Q$-linear map is
linear or, more strongly, a lattice homomorphism.

Before stating the next result, we recall some vector-lattice notions of convergence. Let~$X$ be a vector lattice. If~$(u_n)_{n\in\N}$ is a decreasing sequence in~$X$ such that~$\inf_{n\in\N}u_n=u\in X$, we write~$u_n\downarrow u$. We say that a sequence~$(x_n)_{n\in\N}$ in~$X$ is~$\sigma$-order convergent to~$x\in X$ if there exists a sequence~$(u_n)_{n\in\N}$ in~$X$ such that~$u_n\downarrow 0$ and~$|x-x_n|\leq u_n$ for all~$n\in \N$. In this case, we write~$x_n\xto{\sigma o} x$. It is well known that~$\sigma$-order limits are unique.

We say that~$X$ is \emph{Archimedean} if, for every~$x\in X_+$, we have~$\frac{1}{n}x\downarrow 0$. It is a standard fact that~$X$ is Archimedean if and only if~$\sigma$-order convergence is linear, i.e., both addition and scalar multiplication are continuous.

The following proposition provides several criteria under which a
$\Q$-linear map taking values in a vector lattice is linear.

\begin{prop}
\label{prop:lin-MI}
Let~$X$ be a vector space, let~$Y$ be a vector lattice, and let
$T\colon X\to Y$ be a~$\Q$-linear map. Then the following statements
are equivalent:
\begin{enumerate}
\item\label{item:lin-MI-lin}~$T$ is linear;

\item\label{item:lin-MI-hom} for every~$x\in X$ and~$\lambda \in \R$
$$
\bigl|T(\lambda x)\bigr|
=
|\lambda|\,|T(x)|.$$
\end{enumerate}
If~$Y$ is Archimedean, then these statements are also equivalent to the following:
\begin{enumerate}[resume]
\item\label{item:lin-MI-cont} for every~$x\in X$ and every sequence~$(\lambda_n)$ in
$\R$ satisfying~$\lambda_n\to 0$, one has
$$
\bigl|T(\lambda_nx)\bigr|\xto{\sigma o}0.
$$
\end{enumerate}
\end{prop}
\begin{proof}
\noindent
\textit{(\ref*{item:lin-MI-lin})$\implies$(\ref*{item:lin-MI-hom})}. It follows immediately from linearity of~$T$.

\smallskip
\noindent
\textit{(\ref*{item:lin-MI-hom})$\implies$(\ref*{item:lin-MI-lin})}. Fix~$x\in X$. For~$\lambda \in \R$, we define~$e_\lambda\deq T(\lambda x)-\lambda T(x)$. We will show~$e_\lambda=0$. First, it is obvious that~$e_0=0$. Next, using (\ref*{item:lin-MI-hom}), if~$|\lambda| >1$ we have
$$|\lambda T(x) + e_\lambda|=|T(\lambda x)|=|\lambda| |T(x)|=|\lambda T(x)|,$$
as well as
\begin{align*}
|(\lambda-1) T(x) + e_\lambda|&=|T(\lambda x)-T(x)|=|T((\lambda-1) x)|\\
&=|(\lambda-1)T(x)|=|(\lambda-1) T(x)|.
\end{align*}
This shows that~$|T(x)+\frac{e_\lambda}{\lambda}|=|T(x)+\frac{e_\lambda}{\lambda-1}|=|T(x)|$. By Proposition \ref{prop:MI-convex} (\ref*{item:MI-convex-zero}) applied to the modulus map on~$Y$, this implies that~$e_\lambda =0$.

The preceding argument can be applied to every vector in~$X$. Hence, if~$\lambda$ satisfies~$0<|\lambda|\leq 1$, choose~$m\in \N$ such that~$|m\lambda| >1$. Thus we have
$$T(\lambda x)=T\left((m\lambda)\frac{x}{m}\right)=m\lambda T\left(\frac{x}{m}\right)=\lambda T(x),$$
so~$T$ is linear.

\smallskip

We now assume that~$Y$ is Archimedean.

\noindent
\textit{(\ref*{item:lin-MI-lin})$\implies$(\ref*{item:lin-MI-cont})}. Fix~$\lambda_n\to 0$ in~$\R$. Using linearity of~$T$, along with the Archimedean
property, we get $$|T(\lambda_nx)| = |\lambda_n||T(x)| \xto{\sigma o} 0,$$
using linearity of~$\sigma$-order convergence.

\smallskip
\noindent
\textit{(\ref*{item:lin-MI-cont})$\implies$(\ref*{item:lin-MI-lin})}. Let~$\lambda \in \R$. Let~$(\lambda_n)$ be a sequence in~$\Q$ with
$\lambda_n\to \lambda$. Then~$T((\lambda_n-\lambda)x)\xto{\sigma o} 0$. On the other hand, $\Q$-linearity gives us
$$T((\lambda_n - \lambda )x)=\lambda_n T(x)-T(\lambda x)\xto{\sigma o} \lambda T(x)-T(\lambda x).$$
By the uniqueness of limits, this implies that~$\lambda T(x)-T(\lambda x)=0$, so
$T$ is linear.
\end{proof}

\begin{remark}
The statement of (\ref*{item:lin-MI-cont}) in Proposition \ref{prop:lin-MI} is a natural continuity condition in the setting of vector lattices, but it is straightforward to extend this idea to much more general types of convergence. In the language of net convergence structures given in \cite{OBrienEtAl2023}, suppose that~$X$ is a vector space, $(Y,\to)$ is a Hausdorff net-convergence vector space (i.e., addition and scalar multiplication are continuous and limits are unique) and~$T:X\to Y$ is a~$\Q$-linear map. Then it is easy to show that~$T$ is linear if and only if we have~$T(\lambda_n x)\to 0$ for all~$x\in X$ whenever~$\lambda_n\to 0$ in~$\R$. Indeed, one can simply replace ``$\xto{\sigma o}$'' with ``$\to$'' in the proof of Proposition \ref{prop:lin-MI}.

This may then be used to show that if~$X$ is also a net-convergence vector space and~$T$ is continuous at any point, then~$T$ must be linear. In particular, this implies that many familiar forms of continuity are sufficient to conclude the linearity of~$T$, for instance, norm continuity, continuity with respect to~$\sigma$-order convergence between Archimedean vector lattices, and continuity between Hausdorff topological vector spaces. This also provides a natural continuity condition that may be used to conclude that the map in Lemma \ref{lem:generalized-Mazur-Ulam-Baker} is affine.
\end{remark}

Proposition \ref{prop:lin-MI} gives a useful tool to check whether a given~$\Q$-linear map is linear. We now investigate when~$\Q$-linear maps are actually lattice homomorphisms. The following consequences will be useful when studying specific functional equations at the end of this section.
\begin{cor}
\label{cor:order-MI}
    Let~$X$ and~$Y$ be vector lattices such that~$Y$ is
    Archimedean. Let~$T:X\to Y$ be a~$\Q$-linear map. Then we have the following.
    \begin{enumerate}
        \item If~$T$ is positive, then~$T$ is linear.
        \item If~$T$ preserves at least one of the lattice operations, then~$T$ is a lattice homomorphism.
    \end{enumerate}
\end{cor}
\begin{proof}
Assume first that~$T$ is positive. The proof of this case essentially follows that of \cite[Theorem 1.10]{AliprantisBurkinshaw2006}, included here for completeness. Let~$x,y\in X$ be such that~$x\geq y$. Then~$x-y\geq 0$, so~$T(x-y)\geq 0$ by positivity. Therefore, $\Q$-linearity gives us
$$0\leq T(x-y)= T(x)-T(y)\implies T(y)\leq T(x),$$
and so~$T$ is order-preserving. Since for every~$x\in X$ we have~$-|x|\leq x\leq |x|$, we get~$-T(|x|)\leq T(x)\leq T(|x|)$, and therefore~$|T(x)|\leq T(|x|)$.

Suppose that~$(\lambda_n)\subseteq \R$ is such that~$\lambda_n \to 0$. For each~$n\in \N$, we may choose~$\mu_n\in \Q$ such that~$|\lambda_n| \leq \mu_n$ and~$\mu_n \downarrow 0$. Then for any~$x\in X$ we obtain

\begin{align*}
0&\leq |T(\lambda_n x)|\leq T(|\lambda_n x|)\leq T(\mu_n |x|)= \mu_nT(|x|)\downarrow 0,
\end{align*}
by the Archimedean property. By Proposition \ref{prop:lin-MI} (\ref*{item:lin-MI-cont}), $T$ is linear.

Now suppose that~$T$ preserves one lattice operation. It is well known that each lattice operation can be expressed in terms of any one of the others using~$\Q$-linear operations. Therefore, if a map between vector lattices is~$\Q$-linear and preserves one lattice operation, it preserves them all. In particular, this implies~$T(|x|)=|T(x)|$, so~$T$ is a positive map, and therefore linear. As a linear map preserving all lattice operations, it is a lattice homomorphism.
\end{proof}

\subsection{Modulus isometries}

Next, we specialize Theorem \ref{thm:general_equation_is_Q_linear} to study maps preserving the modulus of differences. Theorem \ref{thm:moduli_isometry} shows that this vector-valued distance identity is substantially more rigid than an ordinary nonsurjective norm isometry.

We first recall a related result; in \cite[Lemma~2]{CelikErcan2010} it was shown that if~$E$ and~$F$ are sublattices of a common vector lattice and~$T:E\to F$ is a bijection
satisfying
$$
|T(x)-T(y)|=|x-y|,
\qquad x,y\in E,
$$
then~$T$ is affine. One may observe, however, that this setting actually forces~$E=F$. Theorem~\ref{thm:moduli_isometry} therefore strengthens this result by removing the bijectivity assumption and allowing the domain to be a proper sublattice of the codomain. If the domain is an ideal, the normalized map is automatically onto~$X$ and is an involution.

\begin{proof}[Proof of Theorem~\ref{thm:moduli_isometry}.]
    We first assume that~$X\subseteq Y$ is a sublattice.
    Let~$T:X\to Y$ be a map satisfying~$|T(x)-T(y)|=|x-y|$ for all~$x,y\in X$. Since the same holds for the map~$U:=T-T
    (0)$, we may assume that~$T(0)=0$. Such a map~$T$ is injective, and for every~$x\in X$ we have~$|T(x)|=|x|$. Next, note that for every~$x,y\in X$ we have
    $$|T(x)-T(y)| = g(x-y),$$
    with~$g(x)=|x|$. From Theorem \ref{thm:general_equation_is_Q_linear} we get that~$T$ is~$\Q$-linear, and since additionally~$g(x)=g(|x|)$, we also conclude that~$T$ is disjointness preserving. Next, observe that for every~$\lambda \in \R$ and~$x\in X$ we have
    $$|T(\lambda x)| = |\lambda x| = |\lambda| |x| = |\lambda||T(x)|.$$
    Proposition \ref{prop:lin-MI} therefore implies that~$T$ is linear.

    Now assume that~$X$ is an ideal. For every~$x\in X$ this implies~$|T(x)|=|x| \in X$. As~$X$ is an ideal, we must have~$T(x)\in X$ as well, which gives that~$T(X)\subseteq X$. Thus, $T^2$ is well-defined.

    We now show that~$T^2=\operatorname{Id}_X$, which in particular will imply that~$T(X)=X$. We first consider the case~$x\in X_+$. We put~$y_1=(T(x))^+$ and~$y_2=(T(x))^-$. Since~$X$ is a sublattice, both~$y_1$ and~$y_2$ belong to~$X$. Observe that
    $$y_1+y_2=|T(x)|=|x|=x,$$
    and therefore~$T(y_1+y_2)=T(x)$. Moreover,
    using linearity of~$T$, we have
    $$T^2(x) = T(y_1 - y_2) = T(y_1) - T(y_2).$$
    We will show that~$T(y_1) = y_1$ and~$T(y_2) = -y_2$, which will imply
    $$T^2(x) = y_1+y_2= |T(x)| = x.$$
    For this purpose, observe that
    $$y_1-y_2 = T(x) = T(y_1+y_2) = T(y_1) + T(y_2),$$
    or equivalently
    $$y_1-T(y_1) = T(y_2)+y_2.$$
    We claim that~$\left(y_1-T(y_1)\right) \perp \left(T(y_2)+y_2\right)$. First, since~$y_1\perp y_2$, disjointness preservation of~$T$ gives~$T(y_1)\perp T(y_2)$. Next, since~$|T(y_1)|=|y_1|$ and~$|T(y_2)|=|y_2|$, we also get~$T(y_1)\perp y_2$ and~$T(y_2)\perp y_1$. This gives~$\left(y_1-T(y_1)\right) \perp \left(T(y_2)+y_2\right)$, as claimed.

    Since~$y_1-T(y_1)$ and~$T(y_2)+y_2$ are disjoint and equal, this forces~$y_1-T(y_1) = 0=T(y_2)+y_2$, or equivalently~$T(y_1)=y_1$ and~$T(y_2)=-y_2$, finishing the proof in the case~$x\in X_+$.

    For general~$x\in X$, using linearity of~$T^2$ and the fact that~$T^2=\operatorname{Id}_X$ on~$X_+$, we have
    $$T^2(x) = T^2(x^+) - T^2(x^-) = x^+ - x^- = x.$$
\end{proof}

\begin{remark}
Examining the proof of Theorem \ref{thm:moduli_isometry}, we see that we can weaken the requirement that~$X$ is an ideal; we only require that~$X$ satisfies
$$|x|\in X \implies  x\in X, \qquad x\in Y.$$
However, one cannot obtain the same conclusion in Theorem \ref{thm:moduli_isometry} if~$X$ is only a sublattice. Indeed, consider~$Y=\R^2$, $X=\{(t,t): t\in \R\}$, and
$$T:X\to Y, \qquad T(t,t)=(-t,t).$$ It is easy to verify that~$X$ is a sublattice of~$Y$ and that~$T$ satisfies
$$|T(x)-T(y)|=|x-y|, \qquad x,y\in X.$$
Nevertheless, $T(X)=\{(-t,t):t\in \R\}$, so~$T(X)\not\subseteq X$.
\end{remark}

\begin{remark}
\label{rem:modulus_isometry_never_positive}
One may ask whether any of the maps~$U\deq T-T(0)$ appearing in Theorem \ref{thm:moduli_isometry} can be lattice homomorphisms. In fact, this is almost never the case. More generally, the map~$U\deq T-T(0)$ in Theorem \ref{thm:moduli_isometry} is positive if and only if~$U=\operatorname{Id}_X$. Indeed, suppose that~$U$ is positive. Then, for all~$x\in X_+$ we have
$$U(x)=|U(x)|=|x|=x.$$
Thus, for all~$x\in X$ we get
$$U(x)=U(x^+-x^-)=U(x^+)-U(x^-)=x^+-x^-=x.$$
\end{remark}

The next example gives a concrete description of all such modulus isometries on~$C(K)$.

\begin{example}
Let~$K$ be a compact Hausdorff space. We characterize all maps
$T\colon C(K)\to C(K)$ satisfying
$$
|T(x)-T(y)|=|x-y|,
\qquad x,y\in C(K).
$$
We first assume that~$T(0)=0$. Taking~$y=0$, we obtain
$$
|T(x)|=|x|,
\qquad x\in C(K).
$$
Since the modulus on~$C(K)$ is defined pointwise, it follows that
$$
|T(x)(t)|=|x(t)|,
\qquad x\in C(K),\quad t\in K.
$$

We first show that, for each fixed~$t\in K$, the value~$T(x)(t)$
depends only on~$x(t)$. Indeed, if~$x,y\in C(K)$ satisfy~$x(t)=y(t)$,
then
    $$0=|x(t)-y(t)|=|T(x)(t)-T(y)(t)|,$$
and hence
    $$T(x)(t)=T(y)(t).$$

For each~$t\in K$, we may therefore define a function
$\varphi_t\colon\R\to\R$ by
$$
\varphi_t(a):=T(x)(t),
$$
where~$x\in C(K)$ is any function satisfying~$x(t)=a$. The preceding
argument shows that~$\varphi_t$ is well defined. Moreover,
$$
|\varphi_t(a)-\varphi_t(b)|=|a-b|,
\qquad a,b\in\R,
$$
and~$\varphi_t(0)=0.$ Thus, $\varphi_t$ is an isometry of~$\R$ fixing the origin. Consequently,
either
$$
\varphi_t(a)=a
\qquad\text{for every }a\in\R,
$$
or
$$
\varphi_t(a)=-a
\qquad\text{for every }a\in\R.
$$
It follows that there exists a function~$\varepsilon\colon K\to\{-1,1\}$
such that
$$
T(x)(t)=\varepsilon(t)x(t),
\qquad x\in C(K),\quad t\in K.
$$
The function~$\varepsilon$ is continuous because
$$
\varepsilon=T(\mathbf{\ind})\in C(K),
$$
where~$\ind$ denotes the constant $1$ function. Therefore,
$$
T(x)=\varepsilon x,
\qquad x\in C(K),
$$
for some continuous function
$\varepsilon\colon K\to\{-1,1\}$. In the general case when~$T$ does not necessarily fix the origin, we immediately get that
$$T(x)=h+\varepsilon x,$$
where~$h=T(0)\in C(K)$ and~$\varepsilon:K\to \{-1,1\}$ is continuous.

Conversely, every~$h\in C(K)$ and every continuous function
$\varepsilon\colon K\to\{-1,1\}$ defines a map
$$
T(x):=h+\varepsilon x
$$
satisfying
    $$|T(x)-T(y)|=|\varepsilon(x-y)|=|x-y|.$$
Hence these are precisely the modulus isometries of~$C(K)$.
\end{example}

\subsection{A lattice-homomorphism equation}

We finish this section with a related equation in which the modulus of the difference~$T(x)-T(y)$ is itself obtained by applying~$T$ to~$|x-y|$. In contrast to Theorem \ref{thm:moduli_isometry}, it characterizes lattice homomorphisms whenever the codomain is Archimedean: the forward implication is given below, while the converse follows immediately from the definition of a lattice homomorphism.

\begin{cor}
\label{cor:T|x-y|_is_lattice_homomorphism}
Let~$X$ and~$Y$ be vector lattices, and let~$T\colon X\to Y$ satisfy
\[
|T(x)-T(y)|=T(|x-y|),
\qquad x,y\in X.
\]
Then~$T$ is a~$\Q$-linear map that preserves all lattice operations. Furthermore, if~$Y$ is also Archimedean, then~$T$
is a lattice homomorphism.
\end{cor}

\begin{proof}
    Let~$T$ be a map satisfying~$|T(x)-T(y)|=T(|x-y|)$ for all~$x,y\in X$.
    We first note that
    $$0 = |T(0)-T(0)| = T(|0|)=T(0).$$
    Next, observe that for every~$x,y\in X$ we have
    $$|T(x)-T(y)|=g(x-y),$$
    where~$g(x)=T(|x|).$ We note that~$g:X\to Y_+$, since for every~$x\in X$ we have~$T(|x|)=|T(x)|\in Y_+$. Thus from Theorem \ref{thm:general_equation_is_Q_linear} it follows that~$T$ is~$\Q$-linear. Moreover, as~$T$ preserves modulus, it preserves all lattice operations.

    If in addition~$Y$ is Archimedean, Corollary \ref{cor:order-MI} guarantees that~$T$ is a lattice homomorphism.
\end{proof}

The Archimedean assumption on the codomain in
Corollary~\ref{cor:T|x-y|_is_lattice_homomorphism} cannot be omitted, as
the following example shows.

\begin{example}
\label{ex:Q_linear_not_linear_pos}
As in Remark \ref{re:Q_linear_not_linear}, let~$f\colon\R\to\R$ be a~$\Q$-linear function that is not linear over
$\R$. Let~$\Rlex$ denote~$\R^2$ equipped with the lexicographic order:
$$
(a,b)\leq (c,d) \quad \iff \quad a< c, \quad \text{ or } \quad a=c ~~\text { and }~~ b\leq d.
$$
It is easy to see that the vector lattice~$\Rlex$ is not Archimedean. Following \cite[Example 1.11]{AliprantisBurkinshaw2006}, we define
$$
T\colon\R\to\Rlex,
\qquad
T(x):=(x,f(x)).
$$
It is immediate that~$T$ is~$\Q$-linear but is not linear over~$\R$.

We claim that~$T$ is positive. Since~$T(0)=0$, fix~$x>0$. Then the first coordinate of~$T(x)$ is positive, and hence~$T(x)=(x,f(x))>0$
in the lexicographic order, proving the claim.

We next show that~$T(|x|)=|T(x)|$ for all~$x\in \R$. If~$x\geq0$, then~$T(x)\geq0$, and therefore
$$
|T(x)|=T(x)=T(|x|).
$$
If~$x<0$, then~$T(x)<0$, and hence using~$\Q$-linearity of~$T$ we get
    $$|T(x)|=-T(x)=T(-x)=T(|x|).$$
This shows that~$T$ preserves modulus. Finally, combining this with~$\Q$-linearity of~$T$, for all~$x,y\in\R$ we have
    $$|T(x)-T(y)|=|T(x-y)|=T(|x-y|).$$
Thus, $T$ satisfies the functional equation from
Corollary~\ref{cor:T|x-y|_is_lattice_homomorphism}, but it is not linear and therefore is not a lattice homomorphism.
\end{example}

\section{Isometries in normed lattices}
\label{sec:isometries_in_normed_lattices}

In this section, we explore the relationship between midpoint injectivity and monotonicity in normed lattices. Our goal is to show how metric information, together with positivity or order preservation, forces lattice structure. Recall that every norm on a normed lattice~$(X,\|\cdot\|)$ is monotone, namely, it satisfies
$$|x|\leq |y| \implies \|x\|\leq \|y\|, \qquad x,y\in X.$$
The norm is called strictly monotone if, whenever~$|x|\leq |y|$ and~$\|x\|=\|y\|$, one has~$|x|=|y|$.

\subsection{Midpoint injectivity and monotonicity} The preceding property of lattice norms motivates the following extension to general vector-lattice-valued functions.

\begin{definition}
Let~$X$ and~$Y$ be vector lattices, and let~$f\colon X\to Y$. We say
that~$f$ is \emph{monotone} if
$$
|x|\leq |y|
\quad\Longrightarrow\quad
f(x)\leq f(y),
\qquad x,y\in X.
$$
We say that~$f$ is \emph{strictly monotone} if it is monotone and
$$
|x|\leq |y|, \quad f(x)=f(y)
\quad\Longrightarrow\quad
|x|=|y|,
\qquad x,y\in X.
$$
\end{definition}
Observe that every monotone function satisfies
$$
f(x)=f(|x|),
\qquad x\in X.
$$
We begin with a simple criterion for a positive linear map to be a lattice homomorphism.
\begin{lemma}
\label{lem:strictly-monotone-homomorphism}
Let~$X$, $Y$, and~$Z$ be vector lattices. Let~$T\colon X\to Y$ be a positive linear map, and let~$f\colon Y\to Z$ be strictly monotone. If
\begin{align}
\label{eq:f_of_T_is_lattice_hom}
    f\bigl(T(|x|)\bigr)=f\bigl(T(x)\bigr), \qquad x\in X,
\end{align}
then~$T$ is a lattice homomorphism.
\end{lemma}
\begin{proof}
Fix~$x\in X$. Since~$T$ is linear and positive, it preserves order. Using this together with the fact that~$\pm x\leq |x|$, we get~$\pm T(x)\leq T(|x|)$ and therefore~$|T(x)|\leq T(|x|)$. Since~$T(|x|)\geq 0$, we obtain~$|T(x)|\leq |T(|x|)|=T(|x|)$. As~$f$ is monotone and satisfies \eqref{eq:f_of_T_is_lattice_hom}, we get $$f(T(x))=f(|T(x)|)\leq f(T(|x|)) = f(T(x)).$$
Strict monotonicity of~$f$ thus implies~$|T(x)|=T(|x|)$. Hence~$T$ preserves one lattice operation and therefore it preserves all of them. As~$T$ is linear, it is a lattice homomorphism.
\end{proof}

\begin{remark}\label{rem:lattice-subspace}
Lemma \ref{lem:strictly-monotone-homomorphism} allows us to recover the following result, originally proved in \cite[Proposition~2]{Miyajima1983} and independently in \cite[Lemma~5]{Abramovich1988}. Let~$(Y,\norm{\cdot})$ be a normed lattice whose norm is strictly monotone. Let~$X$ be a vector subspace of~$Y$ that is itself a vector lattice under the order inherited from~$Y$; such spaces are often referred to as \emph{lattice subspaces} (for instance, see \cite[\S5.4]{AbramovichAliprantis2002}). Crucially, the lattice structure on~$X$ may not be the same as the lattice structure on~$Y$. However, if the restriction of~$\norm{\cdot}$ to~$X$ is also a lattice norm with respect to the lattice structure of~$X$, then Lemma \ref{lem:strictly-monotone-homomorphism} implies that the two lattice structures must agree, i.e.,~$X$ is a sublattice of~$Y$.

Indeed, the inclusion map~$T:X\hookrightarrow Y$ is a positive linear isometry. Therefore, for any~$x\in X$ we have
$$\|T(|x|_X)\|_Y = \||x|_X\|_X = \|x\|_X = \|T(x)\|_Y.$$
Thus, by Lemma \ref{lem:strictly-monotone-homomorphism}, $T$ is a lattice homomorphism. The image of a lattice homomorphism is always a sublattice, so~$X$ must be a sublattice of~$Y$.
\end{remark}

The interplay between convexity and monotonicity has been extensively
studied in the setting of Banach lattices; see the survey
\cite{ForalewskiEtAl2016} and the references therein. Several of these
results extend naturally to midpoint injective functions on vector lattices.

We next study midpoint injectivity on the positive cone. We say that a map~$f:X\to Y$ between vector lattices~$X$ and~$Y$ is midpoint injective on~$X_+$ if for every~$x,y\in X_+$, we have
$$f(x)=f(y)=f\left(\frac{x+y}{2}\right) \implies x=y.$$
The following lemma is known for norms on Banach lattices; see \cite[Theorem~1(i)]{HudzikEtAl2000}. For completeness, we prove its extension to monotone functions on arbitrary vector lattices.

\begin{lemma}
\label{lem:strictly_monotone}
Let~$X$ and~$Y$ be vector lattices, and let~$f\colon X\to Y$ be
monotone. If~$f$ is midpoint injective on~$X_+$, then~$f$ is strictly
monotone.
\end{lemma}
\begin{proof}
Let~$x,y\in X$ such that~$|y|<|x|$. As~$f$ is monotone, we have~$f(y)\leq f(x)$. We need to show that~$f(y)<f(x)$, so by way of contradiction, suppose~$f(y)=f(x)$. Then we have
$$|y|<\frac{|y|+|x|}{2}<|x|,$$
and so by monotonicity we get
$$f(y)=f(|y|)\leq f\left(\frac{|y|+|x|}{2}\right)\leq f(|x|)=f(x).$$
Thus, $f(|y|)= f\left(\frac{|y|+|x|}{2}\right)= f(|x|)$. As~$f$ is midpoint injective on~$X_+$ by assumption, we get~$|x|=|y|$, which is a contradiction.
\end{proof}
Thus every monotone function that is midpoint injective on the positive cone is strictly monotone. In general, however, the converse fails. For instance, consider the normed lattice~$L_1[0,1]$. It is easy to see that the norm on~$L_1[0,1]$ is strictly monotone but not midpoint injective on the positive cone, as we have
$$\norm{\ind_{[0,1/2)}}_{L_1}=\norm{\ind_{[1/2,1]}}_{L_1}=\frac{1}{2}\norm{\ind_{[0,1]}}_{L_1}=\frac{1}{2}.$$
It is natural to ask whether midpoint injectivity on the positive cone extends to the whole space for monotone functions. The following proposition shows that this is indeed the case under the additional assumption of quasiconvexity.

\begin{prop}\label{prop:MI-on-the-cone}
Let~$X$ and~$Y$ be vector lattices, and let~$f\colon X\to Y$ be
monotone and quasiconvex. Then~$f$ is midpoint injective if and only if
$f$ is midpoint injective on~$X_+$.
\end{prop}
We note that in \cite[Theorem~2]{HudzikEtAl2000}, Proposition \ref{prop:MI-on-the-cone} was proved in the special case when~$f$ is a lattice norm on a Banach function space. In particular, Proposition \ref{prop:MI-on-the-cone} extends this result to general normed lattices.
\begin{proof}
Midpoint injectivity on~$X$ implies midpoint injectivity on~$X_+$, so it suffices to prove the converse. Assume that~$f$ is midpoint injective on~$X_+$ and let~$x,y\in X$ be such that~$f(x)=f(y)=f\left(\frac{x+y}{2}\right)$. First note that we have~$\left|\frac{x+y}{2}\right|\leq \frac{|x|+|y|}{2}$, so monotonicity and quasiconvexity give us
\begin{equation}\label{eq:MI-on-the-cone}
\begin{aligned}
f\left(\frac{x+y}{2}\right)&=f\left(\left|\frac{x+y}{2}\right|\right)\leq f\left(\frac{|x|+|y|}{2}\right)\leq f(|x|)\vee f(|y|)\\
&=f(x)\vee f(y)=f\left(\frac{x+y}{2}\right).
\end{aligned}
\end{equation}
Since the first and last terms in the preceding chain are equal, all the inequalities in \eqref{eq:MI-on-the-cone} are equalities. Hence~$f\bigl(\frac{1}{2}(|x|+|y|)\bigr)=f\bigl(\frac{1}{2}(x+y)\bigr)=f(|x|)=f(|y|)$. Since by assumption~$f$ is midpoint injective on~$X_+$, we get~$|x|=|y|$. Moreover, \eqref{eq:MI-on-the-cone} also gives  $f\bigl(\frac{1}{2}(|x+y|)\bigr)=f\bigl(\frac{1}{2}(|x|+|y|)\bigr)$, and since~$f$ is strictly monotone by Lemma \ref{lem:strictly_monotone}, we get that~$\frac{|x|+|y|}{2} = \frac{|x+y|}{2}$. Combining these two facts gives us
$$|x|=|y|=\frac{|x+y|}{2}.$$
By Lemma \ref{lem:modulus_is_midpoint_injective}, the modulus is midpoint injective, which implies that~$x=y$, and so~$f$ is midpoint injective as well.
\end{proof}

The following two examples show that neither of the assumptions in Proposition \ref{prop:MI-on-the-cone} can be removed.

\begin{example}
    On every vector lattice, the map~$x\mapsto x^+$ fails to be monotone since for~$0<x\in X$ we have~$|x|=|-x|$ and yet~$x^+ \neq 0=(-x)^+$. On the other hand, as shown in Example \ref{ex:positive-part-function}, it is convex and thus quasiconvex. Moreover, since the positive part is the identity on~$X_+$, it is midpoint injective on~$X_+$. Nevertheless, as we already observed, it fails to be midpoint injective on~$X$.
\end{example}

\begin{example}
Let~$f:\R^2\to \R$ be defined by~$f(x,y)=(1+|x|)(1+|y|)$. The function~$f$ is monotone with respect to the coordinate-wise order. It is not quasiconvex as we have
$$f\left(\frac{(2,0)+(0,2)}{2}\right)=f((1,1))=4>3=f((2,0))\vee f((0,2)).$$
Let us verify that~$f$ is midpoint injective on~$\R^2_+$. Fix~$a,b,c,d\geq 0$ and suppose that~$f(a,b)=f(c,d)=f\left(\frac{a+c}{2},\frac{b+d}{2}\right)$. Set
$$a' = 1+a, \quad b'=1+b, \quad c'=1+c, \quad d'=1+d.$$
Then~$a'b' = c'd'$, and we have
$$a'b' =\frac{(a'+c')(b'+d')}{4}=\frac{a'b'(1+\frac{c'}{a'})(1+\frac{d'}{b'})}{4}\implies 4=\left(1+\frac{c'}{a'}\right)\left(1+\frac{d'}{b'}\right).$$
Put~$r=\frac{c'}{a'}>0$. From~$a'b'=c'd'$ we have~$\frac{d'}{b'}=\frac{1}{r}$. Hence
$$4=(1+r)\left(1+\frac{1}{r}\right) = 2+r+\frac{1}{r}.$$
Thus~$r+\frac{1}{r}=2$, so~$r=1$. Therefore~$a'=c'$ and~$b'=d'$, and hence~$a=c$ and~$b=d$.

On the other hand, we have
$$f(5,0)=f(-1,2)=f\left(\frac{(5,0)+(-1,2)}{2}\right)=f(2,1)=6,$$
so~$f$ is not midpoint injective.
\end{example}

\begin{remark}
\label{rem:positive-cone-generalized-baker}
    Proposition \ref{prop:MI-on-the-cone} also provides a positive-cone formulation of the midpoint injectivity assumption appearing in our abstract generalization of Baker's theorem. Indeed, consider the setting of Lemma~\ref{lem:generalized-Mazur-Ulam-Baker}, and suppose additionally that~$Y$ and~$Z$ are vector lattices and that~$f\colon Y\to Z$ is monotone and quasiconvex. By Proposition~\ref{prop:MI-on-the-cone}, the assumption that~$f$ is midpoint injective may equivalently be replaced by the requirement that~$f$ be midpoint injective on~$Y_+$. Thus, for monotone quasiconvex functions, the midpoint injectivity needed to obtain the~$\Q$-affinity conclusion of Lemma~\ref{lem:generalized-Mazur-Ulam-Baker} may be verified solely on the positive cone.
\end{remark}

\subsection{Positive isometries} We now turn to positive isometries between normed lattices. Our first goal is to prove Proposition \ref{prop:Abramovich-Baker-version}. It shows that the affinity and disjointness-preserving conclusions of Theorem \ref{thm:Veksler} remain valid when strict monotonicity of the codomain is replaced by the stronger assumption of strict convexity.

\begin{proof}[Proof of Proposition \ref{prop:Abramovich-Baker-version}]
    Let~$T:X\to Y$ be a positive isometry. Since~$Y$ is strictly convex, Theorem \ref{thm:Baker} implies that~$T$ is affine. We claim that~$T-T(0)$ is positive. Let~$x\geq 0$. As~$T$ is positive and affine, for any~$n\in \N$ we have
    $$0\leq T(nx)=n[T(x)-T(0)]+T(0).$$
    Since every normed lattice is Archimedean and~$T(0)\geq 0$, we have~$\frac{1}{n}T(0)\downarrow 0$. This gives us
    $$0\leq T(x)-T(0)+\frac{1}{n}T(0)\downarrow T(x)-T(0),$$
    which proves that~$T-T(0)$ is positive.

    Since the norm on~$Y$ is midpoint injective, we can apply Lemma \ref{lem:strictly_monotone} with~$f=\|\cdot\|_Y$ to conclude that it is also strictly monotone. Finally, since~$T-T(0)$ is a positive linear isometry, for every~$x\in X$,
    $$\|(T-T(0)(|x|)\|_Y = \|x\|_X = \|(T-T(0))(x)\|_Y.$$
    Applying Lemma~\ref{lem:strictly-monotone-homomorphism} with~$f=\|\cdot\|_Y$ implies that~$T-T(0)$ is a lattice homomorphism.
\end{proof}

\begin{remark}
To preserve the readability of these results and emphasize their connection with Abramovich's work, we have chosen to restrict Proposition \ref{prop:Abramovich-Baker-version} and Theorem \ref{thm:nonlinear-Veksler} to normed lattices. We note, however, that these theorems may be extended to a much more general setting. For example, combining Lemma~\ref{lem:generalized-Mazur-Ulam-Baker}, Corollary~\ref{cor:order-MI}, and Lemma~\ref{lem:strictly_monotone} yields a version of Proposition~\ref{prop:Abramovich-Baker-version} for Archimedean vector lattices. In this formulation, the norms on~$X$ and~$Y$ are replaced by monotone dyadically homogeneous functions, with the function on~$Y$ assumed midpoint injective.
\end{remark}

It is natural to ask whether strict monotonicity of the codomain norm alone is sufficient to ensure that the normalized map~$T-T(0)$ of a positive isometry is a lattice homomorphism. The following example shows that this is not the case, and that~$T-T(0)$ need not even be disjointness preserving. In particular, it shows that the affinity assumption is necessary in Theorem \ref{thm:Veksler}.

\begin{example}\label{ex:Veksler-linear}
    Equip~$\R^2$ with the norm given by~$\|(a,b)\|_*=2\max\{|a|,|b|\}$. Consider normed lattices~$X=(\R^2,\|\cdot\|_*)$ and~$Y=(\R^4,\|\cdot\|_1)$, both with respect to the coordinate-wise order. We note that the norm on~$Y$ is strictly monotone, but~$Y$ fails to be strictly convex. Let~$T:X\to Y$ be defined via
    $$T(a,b)=((a+b)^+,(a+b)^-,(a-b)^+,(a-b)^-), \qquad (a,b)\in X.$$
    Note that~$T(0)=0$, $T$ is not linear, and~$T(X)\subseteq Y_+$. In particular, $T$ is positive.

    We claim that~$T$ is an isometry. Fix~$x=(a,b)$ and~$y=(c,d)$ in~$X$. It is straightforward to verify that for every~$r,s\in \R$ we have
    \begin{align}
    \label{eq:identity_for_example_59_1}
        |r^+-s^+|+|r^--s^-|=|r-s|,
    \end{align}
    and
    \begin{align}
    \label{eq:identity_for_example_59_2}
        |r+s|+|r-s|=2\max\{|r|,|s|\}.
    \end{align}
    Applying \eqref{eq:identity_for_example_59_1} with~$r=a+b$ and~$s=c+d$, and \eqref{eq:identity_for_example_59_2} with~$r=a-c$ and~$s=b-d$ gives
    \begin{align*}
    \|T(x)-T(y)\|_Y
    &= |(a+b)^+ - (c+d)^+| + |(a+b)^--(c+d)^-| \\
    &+|(a-b)^+-(c-d)^+|+|(a-b)^--(c-d)^-|\\
    &=|(a+b)-(c+d)| + |(a-b)-(c-d)|\\
    &=|(a-c) + (b-d)| + |(a-c) - (b-d)|\\
    &=2\max\{|a-c|,|b-d|\}\\
    &=\|x-y\|_*.
    \end{align*}
    Finally, despite being a positive isometry, $T$ is not disjointness preserving. Indeed, consider~$x=(1,0)$ and~$y=(0,1)$. Then~$x\perp y$ and yet
    $$|T(x)|\wedge |T(y)|=|(1,0,1,0)|\wedge |(1,0,0,1)|=(1,0,0,0).$$
\end{example}

\subsection{Order-preserving isometries and nonlinear lattice structure} The affinity assumption in Theorem \ref{thm:Veksler} cannot be removed under positivity alone. Theorem \ref{thm:nonlinear-Veksler} shows, however, that positivity and affinity may be replaced by order preservation. We begin with a lattice-theoretic lemma that isolates the nonlinear structure forced by preservation of suprema and infima.

\begin{lemma}
\label{lem:equivalence_for_nonlinear_veksler}
  Let~$X$ and~$Y$ be vector lattices, and let~$T:X\to Y$ satisfy~$T(0)=0$ and preserve suprema and infima. Then~$T$ is disjointness preserving and disjointly additive. Moreover, $T$ preserves all lattice operations if and only if it is odd.
\end{lemma}

We note that related statements have been explored in various settings. Disjointly additive maps preserving infima and suprema on the positive cone were studied in \cite{Abasov2017}, while disjointness preservation within the class of disjotinly additive maps was investigated in \cite{Mykhaylyuk2021}.

\begin{proof}
    We first obtain two useful identities. Since~$T$ preserves infima and suprema, for every~$x\in X$ we have
    \begin{equation}\label{eq:pos-part}
    T(x^+) = T(x \vee 0) = T(x)\vee T(0) = T(x)\vee 0 = [T(x)]^+,
    \end{equation}
    as well as
    \begin{equation}\label{eq:neg-part}
    \begin{aligned}
        -T(-x^-) &= -T(-[(-x) \vee 0]) = -T(x \wedge 0) \\
        &= -(T(x) \wedge T(0)) = -(T(x) \wedge 0) = (-T(x)) \vee 0 = (T(x))^-.
    \end{aligned}
    \end{equation}
    In particular, \eqref{eq:pos-part} shows that~$T$ preserves positive parts. Next we show that~$T$ is disjointness preserving. Fix~$x,y\in X$ with~$x\perp y$. We show that~$T(x)\perp T(y).$ For this purpose, we claim that~$[T(x)]^\delta\perp [T(y)]^\varepsilon$ for any choice of~$\delta,\varepsilon\in \{+,-\}$.

    First, since~$x^+\wedge y^+=0,$ preservation of infima and positive parts gives
    $$[T(x)]^+ \wedge [T(y)]^+=T(x^+)\wedge T(y^+) = T(0)=0,$$
    and therefore~$[T(x)]^+\perp [T(y)]^+.$

    Next, since~$(-x^-)\vee (-y^-) = -(x^-\wedge y^-)=0,$ preservation of suprema gives $$T(-x^-)\vee T(-y^-) = T(0)=0.$$
    Using the identity \eqref{eq:neg-part} we get
        $$[T(x)]^-\wedge [T(y)]^- = -T(-x^-) \wedge -T(-y^-)=-\left(T(-x^-)\vee T(-y^-)\right)=0,$$
    hence~$[T(x)]^-\perp [T(y)]^-.$

    Finally, it suffices to show that~$[T(x)]^+ \perp [T(y)]^-$, as the other case follows symmetrically. Observe that, since~$x^+\perp y^-$, we have
    $$(x^+-y^-)^+ = x^+, \qquad (x^+-y^-)^-=y^-.$$
    Put~$z=x^+-y^-$. Then using preservation of positive parts gives
    $$[T(z)]^+ = T(x^+).$$
    On the other hand, using identity \eqref{eq:neg-part} gives
    $$[T(z)]^-=-T(-z^-) = -T(-y^-).$$
    Since positive and negative parts of any vector are disjoint, this gives~$T(x^+)\perp -T(-y^-).$ Using the fact that~$T$ preserves positive parts together with \eqref{eq:neg-part}, we obtain
    $$[T(x)]^+ \perp [T(y)]^-.$$
    This proves the claim. Finally, since
    $$|T(x)| = [T(x)]^+ + [T(x)]^-, \qquad |T(y)| = [T(y)]^+ + [T(y)]^-,$$
    as all the vectors are pairwise disjoint, we get~$T(x)\perp T(y)$.

    We now show that~$T$ is disjointly additive. Fix~$x,y\in X$ with~$x\perp y$. Since we have~$x^+\perp y^+$, this implies that~$x^++y^+ = x^+\vee y^+$. Using \eqref{eq:pos-part} and preservation of suprema, we obtain
    \begin{equation}\label{eq:pos-dis-add}
    [T(x+y)]^+ = T((x+y)^+) = [T(x)]^+ \vee [T(y)]^+=[T(x)]^+ + [T(y)]^+,
    \end{equation}
    where the last equality follows from disjointness preservation and \eqref{eq:disjointness-fact2}. Similarly, because~$x^-\perp y^-$, we have
    $$-(x^-+y^-) = -(x^-\vee y^-) = (-x^-)\wedge (-y^-).$$
    Using \eqref{eq:neg-part} and preservation of infima, we obtain
    \begin{equation}\label{eq:neg-dis-add}
    [T(x+y)]^-= -T(-(x+y)^-) = [T(x)]^- \vee [T(y)]^-=[T(x)]^- + [T(y)]^-.
    \end{equation}
    Thus \eqref{eq:pos-dis-add} and \eqref{eq:neg-dis-add} imply that
    \begin{align*}
        T(x+y)&=[T(x+y)]^+ - [T(x+y)]^-\\
        &=[T(x)]^++[T(y)]^+ - [T(x)]^- - [T(y)]^-\\
        &=T(x)+T(y),
    \end{align*}
    so~$T$ is disjointly additive.

    We now prove the final equivalence. Suppose first that~$T$ is odd. As~$T$ preserves infima, suprema, and positive parts, we only have to show that it preserves negative parts and modulus. Indeed, the preservation of negative parts follows immediately from oddness and \eqref{eq:neg-part}. Moreover, for any~$x\in X$,
        $$T(|x|) = T(x^+ \vee x^-) = T(x^+) \vee T(x^-) = [T(x)]^+ \vee [T(x)]^- = |T(x)|.$$

    Suppose now that~$T$ preserves all lattice operations. Fix~$x\in X$. Using the fact that~$T$ preserves positive and negative parts, we get~$[T(-x)]^-=[T(x)]^+$ as well as~$[T(-x)]^+=[T(x)]^-$. Thus we have
    $$T(-x) = [T(-x)]^+ - [T(-x)]^- = [T(x)]^- - [T(x)]^+ = -T(x).$$
    Therefore, $T$ is odd.
\end{proof}

We are now in a position to prove Theorem \ref{thm:nonlinear-Veksler}.

\begin{proof}[Proof of Theorem \ref{thm:nonlinear-Veksler}.]
First note that~$T=S-S(0)$ is an order-preserving isometry satisfying~$T(0)=0$. We will make use of the following known identity (see \cite[Theorem 1.7 (3)]{AliprantisBurkinshaw2006}): for any~$u,v\in X$ we have
\begin{equation}\label{eq:Veksler-identity}
|u-v|= u\vee v-u\wedge v.
\end{equation}

We now show that~$T$ preserves suprema and infima. Once this is established, the remaining conclusions follow from Lemma \ref{lem:equivalence_for_nonlinear_veksler}. For any~$x,y\in X$, order preservation gives us
\begin{equation}
\label{eq:two_expressions_positive_and_negative}
  T(x\wedge y)\leq T(x)\wedge T(y)\leq T(x)\vee T(y)\leq T(x\vee y).
\end{equation}
Then \eqref{eq:Veksler-identity} gives us
\begin{equation}
\label{eq:two_expressions_positive}
|T(x)-T(y)|= T(x)\vee T(y)-T(x)\wedge T(y) \leq T(x\vee y)-T(x\wedge y).
\end{equation}
However, the isometry condition gives
$$\norm{T(x\vee y)-T(x\wedge y)}_Y=\norm{x\vee y - x\wedge y}_X=\norm{x-y}_X=\norm{T(x)-T(y)}_Y,$$
where again we have made use of \eqref{eq:Veksler-identity}. Then strict monotonicity implies that $$|T(x)\vee T(y)-T(x)\wedge T(y)| = |T(x\vee y)-T(x\wedge y)|.$$
Since both of these expressions are positive due to \eqref{eq:two_expressions_positive}, rearranging gives
\begin{equation}\label{eq:negative-positve-zero}
  T(x)\vee T(y) - T(x\vee y)= T(x)\wedge T(y)-T(x\wedge y).
\end{equation}
Using \eqref{eq:two_expressions_positive_and_negative}, we note that the left-hand side of \eqref{eq:negative-positve-zero} is negative, while the right-hand side is positive, so both terms must be zero. Thus, $T$ preserves suprema and infima. Using Lemma \ref{lem:equivalence_for_nonlinear_veksler}, we conclude the proof.
\end{proof}

Our final example shows that linearity is not guaranteed by Theorem \ref{thm:nonlinear-Veksler}.
\begin{example}\label{ex:nonlinear-Veksler}
Let~$\varphi:\R \to \R$ be~$1$-Lipschitz, i.e., $\varphi$ satisfies
\begin{equation}\label{eq:Lipschitz}
  |\varphi(y)-\varphi(x)|\leq |y-x|, \qquad x,y\in \R.
\end{equation}
Consider the map~$T:L_1[0,1]\to L_1[0,1]$ defined by
$$T(f)(x)\deq \begin{cases}
f(2x)+\varphi(f(2x)),& x\in \left[0,\frac{1}{2}\right],\\
f(2x-1)-\varphi(f(2x-1)),& x\in \left(\frac{1}{2},1\right].
\end{cases}$$
We first note that~$T$ is well defined, since for any~$x\in \R$ we have
$$|\varphi(x)|\leq |\varphi(0)|+|x|,$$
and so~$\varphi\circ f\in L_1[0,1]$ whenever~$f\in L_1[0,1]$. We show that~$T$ is an order-preserving isometry.

First, consider~$f,g\in L_1[0,1]$ with~$f\leq g$. For~$x\in \left[0,\frac{1}{2}\right]$, \eqref{eq:Lipschitz} gives us
\begin{align*}
T(g)(x)-T(f)(x)&=g(2x)+\varphi(g(2x))-f(2x)-\varphi(f(2x))\\
&= g(2x)-f(2x)+\varphi(g(2x))-\varphi(f(2x))\\
&\geq g(2x)-f(2x)-|\varphi(g(2x))-\varphi(f(2x))|\\
&\geq g(2x)-f(2x)-|g(2x)-f(2x)|=0,
\end{align*}
since~$g\geq f$. Therefore~$T(g)(x)\geq T(f)(x)$. Using the same reasoning, we obtain~$T(g)(x)\geq T(f)(x)$ for~$x\in \left(\frac{1}{2},1\right]$, which lets us conclude that~$T(g)\geq T(f)$, so~$T$ is order-preserving.

Next, we obtain a useful identity. For~$x,y\in \R$, applying  \eqref{eq:identity_for_example_59_2} with~$r=x-y$ and~$s=\varphi(x)-\varphi(y)$ together with
\eqref{eq:Lipschitz} yields
\begin{equation}\label{eq:straightforward-fact}
  |x-y+\varphi(x)-\varphi(y)|+|x-y-[\varphi(x)-\varphi(y)]|=2|x-y|.
\end{equation}
Then for any~$f,g\in L_1[0,1]$, \eqref{eq:straightforward-fact} gives
\begin{align*}
\norm{T(g)-T(f)}_{L_1}&=\int_{0}^{\frac{1}{2}} |g(2x)-f(2x)+\varphi(g(2x))-\varphi(f(2x))|\dx\\
&+ \int_{\frac{1}{2}}^{1} |g(2x-1)-f(2x-1)-\varphi(g(2x-1))+\varphi(f(2x-1))|\dx\\
&=\frac{1}{2}\int_{0}^{1} |g(x)-f(x)+\varphi(g(x))-\varphi(f(x))|\dx\\
&+\frac{1}{2}\int_{0}^{1} |g(x)-f(x)-\varphi(g(x))+\varphi(f(x))|\dx\\
&=\int_{0}^{1} |g(x)-f(x)|\dx=\norm{g-f}_{L_1}.
\end{align*}
Thus~$T$ is also an isometry. Since the norm on~$L_1[0,1]$ is strictly monotone, Theorem \ref{thm:nonlinear-Veksler} implies that~$T-T(0)$ preserves suprema and infima and is therefore both disjointness preserving and disjointly additive.

Furthermore, if~$\varphi$ is chosen to be odd, then~$\varphi(0)=0$, so~$T(0)=0$, and~$T$ itself is odd. Hence Theorem \ref{thm:nonlinear-Veksler} shows that~$T$ preserves all lattice operations.

However, it is easy to see that~$T$ need not be affine. For instance, take~$\varphi(t)\deq \sin(t)$. We note that this gives~$T(0)=0$. If~$\ind$ denotes the constant~$1$ function, then we have~$T(\pi \ind)=\pi \ind$. On the other hand,
$$T\left(\frac{\pi}{2}\ind\right)(x)=\begin{cases}
\frac{\pi}{2}+1,& x\in \left[0,\frac{1}{2}\right],\\
\frac{\pi}{2}-1,& x\in \left(\frac{1}{2},1\right],
\end{cases}$$
so~$2T\left(\frac{\pi}{2}\ind\right)\neq T(\pi\ind)$. Thus~$T$ is a disjointness preserving, disjointly additive isometry that preserves all lattice operations, yet it is not a lattice homomorphism.
\end{example}

\subsection*{Acknowledgments} The authors are grateful to Vladimir Troitsky for many valuable conversations during the development of this paper and for his encouragement in pursuing this topic. We are also grateful to him for bringing \cite{Abramovich1988} to our attention.

The authors acknowledge the use of ChatGPT-5.6 Sol during the preparation of this manuscript for correcting grammar and style, exploring proof ideas, conducting literature searches, and improving the readability of the exposition. The authors assume responsibility for all content.

\bibliographystyle{plain}
\bibliography{bibfile}

@book{AbramovichAliprantis2002,
	address = {Providence, R.I},
	author = {Abramovich, Y. A. and Aliprantis, Charalambos D.},
	booktitle = {An invitation to operator theory},
	isbn = {0821821466},
	language = {eng},
	lccn = {2002074420},
	publisher = {American Mathematical Society},
	series = {Graduate studies in mathematics, v. 50},
	title = {An invitation to operator theory},
	year = {2002}}

@article{OBrienEtAl2023,
	author = {O'Brien, M. and Troitsky, V. G. and van der Walt, J. H.},
	doi = {10.2989/16073606.2021.2012721},
	journal = {Quaestiones Mathematicae},
	number = {2},
	pages = {243--280},
	publisher = {Taylor \& Francis},
	title = {Net convergence structures with applications to vector lattices},
	volume = {46},
	year = {2023}}

@article{Abasov2017,
title = {On extensions of some nonlinear maps in vector lattices},
journal = {Journal of Mathematical Analysis and Applications},
volume = {455},
number = {1},
pages = {516-527},
year = {2017},
issn = {0022-247X},
doi = {https://doi.org/10.1016/j.jmaa.2017.05.063},
url = {https://www.sciencedirect.com/science/article/pii/S0022247X17305231},
author = {Abasov, N. and Pliev, M.}
}

@article{Mykhaylyuk2021,
author={Mykhaylyuk, V. and Pliev, M. and Popov, M.},
title={The lateral order on {Riesz} spaces and orthogonally additive operators},
journal={Positivity},
year={2021},
month={Apr},
day={01},
volume={25},
number={2},
pages={291-327},
issn={1572-9281},
doi={10.1007/s11117-020-00761-x},
url={https://doi.org/10.1007/s11117-020-00761-x}
}

@article{Abramovich1988,
  author  = {Abramovich, Y. A.},
  title   = {Isometries of Normed Lattices},
  journal = {Optimizatsiya},
  volume  = {43(60)},
  pages   = {74--80},
  year    = {1988},
  note    = {In Russian}
}

@article{MartinezLegaz2012,
	author = {Mart{\'\i}nez-Legaz, J. E. and Martin{\'o}n, A.},
	doi = {10.1007/s11750-011-0222-8},
	isbn = {1863-8279},
	journal = {TOP},
	number = {2},
	pages = {503--516},
	title = {On the infimum of a quasiconvex vector function over an intersection},
	volume = {20},
	year = {2012}}

@book{KusraevKutateladze1995,
	author = {Kusraev, A. G. and Kutateladze, S. S.},
	booktitle = {Subdifferentials: theory and applications},
	isbn = {0792333896},
	publisher = {Kluwer Academic Publishers},
    address = {Dordrecht},
	series = {Mathematics and its applications},
	title = {Subdifferentials: theory and applications},
	volume = {323},
	year = {1995}}

@article{Miyajima1983,
	author = {Miyajima, S.},
	doi = {10.14492/hokmj/1381757794},
	journal = {Hokkaido Mathematical Journal},
	number = {1},
	pages = {83--91},
	title = {Structure of {Banach} Quasi-Sublattices},
	volume = {12},
	year = {1983}}

@article{Dong2015,
title = {On approximate isometries and application to stability of a functional equation},
journal = {Journal of Mathematical Analysis and Applications},
volume = {426},
number = {1},
pages = {125-137},
year = {2015},
issn = {0022-247X},
doi = {https://doi.org/10.1016/j.jmaa.2015.01.045},
url = {https://www.sciencedirect.com/science/article/pii/S0022247X15000645},
author = {Dong, Y.},
}

@article{LemmensVanGaans2003,
  author  = {Lemmens, B. and van Gaans, O. W.},
  title   = {Periods of order-preserving nonexpansive maps on strictly convex normed spaces},
  journal = {Journal of Nonlinear and Convex Analysis},
  volume  = {4},
  number  = {3},
  pages   = {353--363},
  year    = {2003}
}

@article{AbasovPliev2018,
  author  = {Abasov, N. and Pliev, M.},
  title   = {Disjointness-preserving orthogonally additive operators in vector lattices},
  journal = {Banach Journal of Mathematical Analysis},
  volume  = {12},
  number  = {3},
  pages   = {730--750},
  year    = {2018},
  doi     = {10.1215/17358787-2018-0001}
}

@article{Feldman2019,
  author  = {Feldman, W. A.},
  title   = {A factorization for orthogonally additive operators on {B}anach lattices},
  journal = {Journal of Mathematical Analysis and Applications},
  volume  = {472},
  number  = {1},
  pages   = {238--245},
  year    = {2019},
  doi     = {10.1016/j.jmaa.2018.11.021}
}

@article{Feldman2013,
title = {Lattice preserving maps on lattices of continuous functions},
journal = {Journal of Mathematical Analysis and Applications},
volume = {404},
number = {2},
pages = {310-316},
year = {2013},
issn = {0022-247X},
doi = {https://doi.org/10.1016/j.jmaa.2013.03.017},
url = {https://www.sciencedirect.com/science/article/pii/S0022247X13002102},
author = {Feldman, W. A.}
}

@article{LeungTangPersistence2016,
  author  = {Leung, D. H. and Tang, W.-K.},
  title   = {Persistence of {Banach} lattices under nonlinear order isomorphisms},
  journal = {Positivity},
  volume  = {20},
  number  = {3},
  pages   = {709--717},
  year    = {2016},
  doi     = {10.1007/s11117-015-0382-0}
}

@article{LeungTang2016,
  author  = {Leung, D. H. and Tang, W.-K.},
  title   = {Nonlinear order isomorphisms on function spaces},
  journal = {Dissertationes Mathematicae},
  volume  = {517},
  pages   = {1--75},
  year    = {2016},
  doi     = {10.4064/dm737-11-2015}
}

@article{MazonSegura1990,
  author  = {Maz{\'o}n, J. M. and Segura de Le{\'o}n, S.},
  title   = {Order bounded orthogonally additive operators},
  journal = {Revue Roumaine de Math{\'e}matiques Pures et Appliqu{\'e}es},
  volume  = {35},
  number  = {4},
  pages   = {329--353},
  year    = {1990}
}

@article{Feldman2017,
  author  = {Feldman, W. A.},
  title   = {A characterization of non-linear maps satisfying orthogonality properties},
  journal = {Positivity},
  volume  = {21},
  number  = {1},
  pages   = {85--97},
  year    = {2017},
  doi     = {10.1007/s11117-016-0408-2}
}

@article{Pliev2021,
  author  = {Pliev, M. A.},
  title   = {On ${C}$-compact orthogonally additive operators},
  journal = {Journal of Mathematical Analysis and Applications},
  volume  = {494},
  number  = {1},
  pages   = {124594},
  year    = {2021},
  doi     = {10.1016/j.jmaa.2020.124594}
}

@article{Ger1993,
  title={On a characterization of strictly convex spaces},
  author={Ger, R.},
  journal={Atti della Accademia delle Scienze di Torino. Classe di Scienze Fisiche, Matematiche e Naturali},
  volume={127},
  number={3-4},
  pages={131-138},
  year={1993}
}

@article{Skof1993,
  title={On the functional equation $\|f(x+y)-f(x)\|=\|f(y)\|$},
  author={Skof, F.},
  journal={Atti Accad. Sci. Torino Cl. Sci. Fis. Mat. Natur},
  volume={127},
  pages={229--237},
  year={1993}
}

@article{Sikorska2005,
title = {Stability of the preservation of the equality of distance},
journal = {Journal of Mathematical Analysis and Applications},
volume = {311},
number = {1},
pages = {209-217},
year = {2005},
issn = {0022-247X},
doi = {https://doi.org/10.1016/j.jmaa.2005.02.039},
url = {https://www.sciencedirect.com/science/article/pii/S0022247X05001423},
author = {Sikorska, J.}
}

@article{CelikErcan2010,
    url = {https://doi.org/10.1515/gmj.2010.040},
    title = {Mazur-{U}lam theorem for {R}iesz spaces},
    author = {{\c C}elik, C. and Ercan, Z.},
    pages = {635--640},
    volume = {17},
    number = {4},
    journal = {Georgian Mathematical Journal},
    doi = {doi:10.1515/gmj.2010.040},
    year = {2010}
}

@article{HudzikEtAl2000,
	author = {Hudzik, H. and Kami{\'n}ska, A. and Masty{\l}o, M.},
	journal = {Rocky Mountain Journal of Mathematics},
	number = {3},
	pages = {933--950},
	title = {Monotonicity and Rotundity Properties in {Banach} Lattices},
	volume = {30},
	year = {2000}}

@incollection{ForalewskiEtAl2016,
    publisher = {Birkhäuser},
	address = {Cham},
	author = {Foralewski, P. and Hudzik, H. and Kowalewski, W. and Wis{\l}a, M.},
	booktitle = {Ordered Structures and Applications},
	isbn = {978-3-319-27842-1},
	pages = {203--232},
    editor = {de Jeu, M. and de Pagter, B. and van Gaans, O. and Veraar, M.},
	title = {Monotonicity Properties of {Banach} Lattices and Their Applications - a Survey},
	year = {2016}}

@book{AliprantisBurkinshaw2006,
	author = {Aliprantis, C. D. and Burkinshaw, O.},
	publisher = {Springer},
	title = {Positive operators},
    address   = {Dordrecht},
	year = {2006}}

@article{Vaisala2003,
	author = {V{\"a}is{\"a}l{\"a}, J.},
	doi = {10.1080/00029890.2003.11920004},
	journal = {The American Mathematical Monthly},
	number = {7},
	pages = {633--635},
	publisher = {Taylor \& Francis},
	title = {A Proof of the {Mazur-Ulam} Theorem},
	volume = {110},
	year = {2003}}

@article{Nica2012,
	author = {Nica, B.},
	doi = {https://doi.org/10.1016/j.exmath.2012.08.010},
	issn = {0723-0869},
	journal = {Expositiones Mathematicae},
	number = {4},
	pages = {397-398},
	title = {The {Mazur-Ulam} theorem},
	url = {https://www.sciencedirect.com/science/article/pii/S0723086912000515},
	volume = {30},
	year = {2012}}

@article{Baker1971,
	author = {J. A. Baker},
	issn = {00029890, 19300972},
	journal = {The American Mathematical Monthly},
	number = {6},
	pages = {655--658},
	publisher = {[Taylor & Francis, Ltd., Mathematical Association of America]},
	title = {Isometries in Normed Spaces},
	url = {http://www.jstor.org/stable/2316577},
	urldate = {2026-03-17},
	volume = {78},
	year = {1971}}

@article{Mazur1932,
	author = {Mazur, S. and Ulam, S.},
	journal = {C. R. Acad. Sci. Paris},
	title = {Sur les transformations isom{\'e}triques d'espaces vectoriels norm{\'e}s},
	volume = {194},
    pages = {946-948},
	year = {1932}}

\end{document}